\documentclass[11pt,a4paper]{article}

\usepackage[nobysame,initials]{amsrefs}
\usepackage{bbm}

\usepackage{epsf,epsfig,amsfonts,amsgen,amsmath,amstext,amsbsy,amsopn,amsthm,cases,listings,color
}
\usepackage{ebezier,eepic}
\usepackage{color}
\usepackage{multirow}
\usepackage{epstopdf}
\usepackage{graphicx}
\usepackage{pgf,tikz}
\usepackage{mathrsfs}
\usepackage[marginal]{footmisc}
\usepackage{enumitem}
\usepackage[titletoc]{appendix}
\usepackage{booktabs}
\usepackage{url}
\usepackage{mathtools}
\usepackage{hyperref}

\usepackage{pgfplots}
\usepackage{authblk}
\usepackage{amssymb}

\usepackage{wasysym}

\usepackage{empheq}

\usepackage{dsfont}

\usepackage{caption}
\usepackage{subcaption}
\pgfplotsset{compat=1.18}
\usepackage{mathrsfs}
\usepackage{wasysym} 
\usetikzlibrary{arrows}
\usepackage{aligned-overset}
\usepackage{bm}

\usepackage[normalem]{ulem}
\usepackage{makecell}

\newcommand{\Mod}[1]{\ (\mathrm{mod}\ #1)}

\allowdisplaybreaks[1]

\definecolor{uuuuuu}{rgb}{0.27,0.27,0.27}
\definecolor{sqsqsq}{rgb}{0.1255,0.1255,0.1255}

\newtheorem{definition}{Definition} [section]
\newtheorem{theorem}[definition]{Theorem}
\newtheorem{lemma}[definition]{Lemma}
\newtheorem{proposition}[definition]{Proposition}

\newtheorem{claim}[definition]{Claim}

\newtheorem{fact}[definition]{Fact}
\newenvironment{pf}{\noindent{\bf Proof}}

\theoremstyle{remark}

\def\qed{\hfill \rule{4pt}{7pt}}

\newcommand{\CF}{\mathcal{C}} % will be C5K4 or C_ell
\newcommand{\CC}{\mathcal{C}'} % will be C5K4 or C7F1F2
\newcommand{\heavy}{h} % for bad edges with at least one heavy pair
\newcommand{\C}[1]{\mathcal{#1}}
\newcommand{\I}[1]{{\mathbbm #1}}

\newcommand{\eval}[1]{[\![#1]\!]} 
\newcommand{\ra}[1]{\cite[#1]{Razborov10}}
\newcommand{\e}{\varepsilon}
\newcommand{\hide}[1]{}

\begin{document}
%%%%%%%%%%%%%%%%%%%%%%%%%%%%%%%%%%%%%%%%%%%%%%%%%%%%%%%
\title{\bf\Large The Tur\'{a}n density of short tight cycles}
\author{Levente Bodn\'ar}
\author{Jared Le\'on}
\author{Xizhi Liu}
\author{Oleg Pikhurko}
\affil{Mathematics Institute and DIMAP,
             University of Warwick,
             Coventry, 
             %CV4 7AL, 
             UK
}
\date{\today}
%%%%%%%%%%%%%%%%%%%%%%%%%%%%%%%%%%%%%%%%%%%%%%%%%%%%
%%%%%%%%%%%%%%%%%%%%%%%%%%%%%%%%%%%%%%%%%%%%%%%%%%%
\maketitle
%\footnote{footnote}
%%%%%%%%%%%%%%%%%%%%%%%%%%%%%%%%%%%%%%%%%%%%%%%%%
%%%%%%%%%%%%%%%%%%%%%%%%%%%
\begin{abstract}
The $3$-uniform \textbf{tight $\ell$-cycle} $C_\ell^{3}$ is the $3$-graph on $\{1,\dots,\ell\}$ consisting of all $\ell$ consecutive triples in the cyclic order. Let 
%a $3$-graph family 
$\CF$ be either the pair $\{C_{4}^{3}, C_{5}^{3}\}$ or the single tight $\ell$-cycle  $C_{\ell}^{3}$ for some $\ell\ge 7$ not divisible by~$3$. 

We show that the \textbf{Tur\'an density} of $\CF$, that is, the asymptotically maximal edge density of a large $\CF$-free $3$-graph, is equal to $2\sqrt{3} - 3$.  We also establish the corresponding Erd\H{o}s--Simonovits-type stability result, informally stating that all almost maximum $\CF$-free graphs are close in the edit distance to a 2-part recursive construction.  This extends the earlier analogous results of Kam{\v c}ev--Letzter--Pokrovskiy~[%
\emph{``The Tur\'an density of tight cycles in three-uniform hypergraphs",}
Int.\ Math.\ Res.\ Not.\ 6 (2024), 4804–4841]
%~\cite{KLP24} 
that apply for sufficiently large $\ell$ only.  

Additionally, we prove a finer structural result that allows us to determine the maximum number of edges in a $\{C_{4}^{3}, C_{5}^{3}\}$-free $3$-graph with a given number of vertices up to an additive $O(1)$ error term.
\end{abstract}
%\medskip
% \textbf{Keywords:} Hypergraphs, Tur\'{a}n problem, tight cycles.
% \medskip
% \textbf{MSC2020:} 05C65, 05C35, 05D99. 

%%%%%%%%%%%%%%%%%%%%%%%%%%%%%
% \tableofcontents
%%%%%%%%%%%%%%%%%%%%%%%%%%%%%

%%%%%%%%%%%%%%%%%%%%%%%%%%%%%%%%%%%%%%%%%%%%%%%%%%%%%%
\section{Introduction}\label{SEC:Intorduction}
% \subsection{Motivation}
Given an integer $r\ge 2$, an \textbf{$r$-uniform hypergraph} (henceforth an \textbf{$r$-graph}) $\mathcal{H}$ is a collection of $r$-subsets of some set $V$. We call $V$ the \textbf{vertex set} of $\C H$ and denote it by $V(\C H)$. When $V$ is understood, we usually identify a hypergraph $\mathcal{H}$ with its set of edges. 
For a subset $S \subseteq V(\mathcal{H})$, we use $\mathcal{H}[S]$ to denote the \textbf{induced subgraph} of $\mathcal{H}$ on $S$ which consists of all edges $e\in\mathcal{H}$ with $e\subseteq S$.

Given a family $\mathcal{F}$ of $r$-graphs, we say an $r$-graph $\mathcal{H}$ is \textbf{$\mathcal{F}$-free}
if it does not contain any member of $\mathcal{F}$ as a subgraph.
The \textbf{Tur\'{a}n number} $\mathrm{ex}(n, \mathcal{F})$ of $\mathcal{F}$ is the maximum number of edges in an $\mathcal{F}$-free $r$-graph on $n$ vertices. 
The \textbf{Tur\'{a}n density} of $\mathcal{F}$ is defined as $\pi(\mathcal{F})\coloneq \lim_{n\to\infty}\mathrm{ex}(n,\mathcal{F})/\binom nr$. 
The existence of this limit follows from the simple averaging argument of Katona--Nemetz--Simonovits~\cite{KNS64} which shows that $\mathrm{ex}(n,\mathcal{F})/\binom{n}{r}$ is non-increasing in $n$.

For $r=2$, the value $\pi(\mathcal{F})$ is well understood thanks to the classical work of Erd\H{o}s--Stone~\cite{ES46} (see also~\cite{ES66}), which extends Tur\'{a}n's seminal theorem~\cite{Tur41}.
For $r \ge 3$, determining $\pi(\mathcal{F})$ is notoriously difficult in general, despite significant effort devoted to this area. 
For results up to~2011, we refer the reader to the excellent survey by Keevash~\cite{Kee11}.

In this work, we focus on the Tur\'{a}n density of $3$-uniform tight cycles. 
For an integer $\ell \ge 4$, the \textbf{tight $\ell$-cycle} $C_{\ell}^{3}$ is the $3$-graph on $[\ell]\coloneqq \{1,\dots,\ell\}$ with edge set 
\begin{align*}
    \big\{\,\{1,2,3\}, \{2,3,4\}, \cdots, \{\ell-2,\ell-1,\ell\},\{\ell-1,\ell,1\},\{\ell,1,2\}\,\big\},
\end{align*}
which consists of all consecutive triples in the cyclic order on $[\ell]$.
The \textbf{tight $\ell$-cycle minus one edge} $C_{\ell}^{3-}$ is the $3$-graph 
obtained from  $C_{\ell}^{3}$ by removing one edge. 
Note that $C_{4}^{3}$ is the complete $3$-graph on $4$ vertices, also denoted by $K_{4}^{3}$. 

If $\ell \equiv 0 \Mod{3}$ (i.e.\ $\ell$ is divisible by $3$) then  both $C_{\ell}^{3}$ and $C_{\ell}^{3-}$ are $3$-partite and thus it holds that $\pi(C_{\ell}^{3}) = \pi(C_{\ell}^{3-})=0$ by the classical general result of Erd\H{o}s~\cite{Erd64KST}.
Very recently, using a sophisticated stability argument combined with results from digraph theory, Kam{\v c}ev--Letzter--Pokrovskiy~\cite{KLP24} proved that $\pi(C_{\ell}^{3}) = 2\sqrt{3}-3$ for all sufficiently large $\ell$ satisfying $\ell \not\equiv 0 \Mod{3}$. 
Later, being partially inspired by~\cite{KLP24}, Balogh--Luo~\cite{BL24C5Minus} proved that $\pi(C_{\ell}^{3-}) = 1/4$ for all sufficiently large $\ell$ satisfying $\ell \not\equiv 0 \Mod{3}$. 
Very recently, Lidick\'y--Mattes--Pfender~\cite{LMP24} and, independently, the authors of the present work~\cite{BLLP24} proved that $\pi(C_{\ell}^{3-}) = 1/4$ for every integer $\ell \ge 5$ satisfying $\ell \not\equiv 0 \Mod{3}$. 

We call a $3$-graph $\mathcal{H}$ \textbf{semi-bipartite} if there exists a partition $V_1 \cup V_2 = V(\mathcal{H})$ such that every edge in $\mathcal{H}$ contains exactly two vertices from $V_1$. 
% \begin{align*}
%     \mathcal{H}
%     = \left\{e \in \binom{V(\mathcal{H})}{3} \colon |e\cap V_1| = 2\right\}. 
% \end{align*}
%
%For convenience, we write $\mathcal{H}[V_1, V_2]$ to indicate the property that $|e\cap V_1| = 2$ for every $e \in \mathcal{H}$.
Denote by $\mathbb{B}[V_1,V_2]$ \textbf{the complete semi-bipartite} $3$-graph with parts $V_1, V_2$, i.e.,
\begin{align*}
    \mathbb{B}[V_1,V_2]
    \coloneqq \left\{e\in \binom{V_1 \cup V_2}{3} \colon |e\cap V_1| = 2\right\}.
\end{align*}

In~{\cite{MR02}}, Mubayi--R\"{o}dl provided the following construction, which establishes the lower bound $\pi(C_{\ell}^{3})\ge 2\sqrt{3} - 3$ for every $\ell \ge 4$ not divisible by $3$.

For $n \in \{0, 1,2\}$, the (unique) $n$-vertex \textbf{$\mathbb{B}_{\mathrm{rec}}$-construction} is the empty $3$-graph on $n$ vertices. 
For $n \ge 3$, an $n$-vertex $3$-graph $\mathcal{H}$ is a \textbf{$\mathbb{B}_{\mathrm{rec}}$-construction} if there exists a partition $V_1 \cup V_2 = V(\mathcal{H})$ into non-empty parts such that $\mathcal{H}$ is obtained from $\mathbb{B}[V_1,V_2]$ by adding a copy of $\mathbb{B}_{\mathrm{rec}}$-construction into $V_2$. 
Additionally, a $3$-graph is called a \textbf{$\mathbb{B}_{\mathrm{rec}}$-subconstruction} if it is a subgraph of some $\mathbb{B}_{\mathrm{rec}}$-construction.

Let $b_{\mathrm{rec}}(n)$ denote the maximum number of edges in an $n$-vertex $\mathbb{B}_{\mathrm{rec}}$-construction. 
It is clear from the definition that, for each $n\ge 3$,
\begin{align*}
    b_{\mathrm{rec}}(n)
    = \max\left\{ \binom{n_1}{2} n_2 + b_{\mathrm{rec}}(n_2) \colon n_1 + n_2 = n~\text{and}~n_1 \ge 1\right\}. 
\end{align*}
Simple calculations show that, as $n\to\infty$, 
the optimal part ratio $n_1/n$ is~$\gamma+o(1)$ and it holds that $b_{\mathrm{rec}}(n)= \left(\alpha+o(1)\right) \binom{n}{3}$, where we define
\begin{align}\label{equ:gamma-alpha-def}
    \alpha \coloneqq 2\sqrt{3} - 3 
    \quad\text{and}\quad 
    \gamma \coloneqq \frac{3-\sqrt{3}}{2}. 
\end{align}

Let us assume for the rest of this paper, that
\begin{align}\label{equ:def-C}
    %\mathrm{either}\quad 
    \CF \coloneqq \{C_{4}^{3}, C_{5}^{3}\} \quad \mathrm{or}\quad \CF \coloneqq\{C_{\ell}^{3}\}\mbox{ for some $\ell\ge 7$ with $\ell\not\equiv 0\Mod3$}.
\end{align}

It is easy to see that every $\mathbb{B}_{\mathrm{rec}}$-construction is $\CF$-free, thus giving that $\pi(\CF) \ge  \alpha$.
%2\sqrt{3} - 3$. 
Our first main result shows that this is in fact equality.

\begin{theorem}\label{THM:turan-density-C5K4} 
    Let $\CF$ be as in~\eqref{equ:def-C}. Then it holds that $\pi(\CF) \le \alpha$.
\end{theorem}

We also establish the Erd{\H o}s--Simonovits-type stability property~\cite{Erdos67a,Sim68} for each case from Theorem~\ref{THM:turan-density-C5K4}. 
\begin{theorem}\label{THM:C5K4-stability}
 Let $\CF$ be as in~\eqref{equ:def-C}.  Then for every $\varepsilon > 0$, there exist $\delta > 0$ and $n_0$ such that every $\CF$-free $3$-graph with $n\ge n_0$ vertices and  at least $\alpha\binom{n}{3} - \delta n^3$ edges is a $\mathbb{B}_{\mathrm{rec}}$-subconstruction after removing at most $\varepsilon n^3$ edges.      
\end{theorem}

Using Theorem~\ref{THM:C5K4-stability} and additional arguments, we establish the following refined structural property for maximum $\{C_{4}^{3}, C_{5}^{3}\}$-free $3$-graphs on sufficiently large vertex sets.  
\begin{theorem}\label{THM:K4C5-first-level-exact}
    For every $\varepsilon > 0$ there exists $n_0$ such that the following holds for every $n \ge n_0$. 
    Suppose that $\mathcal{H}$ is an $n$-vertex $\{C_{4}^{3}, C_{5}^{3}\}$-free $3$-graph with $\mathrm{ex}(n,\{C_{4}^{3}, C_{5}^{3}\})$ edges. Then there exists a partition $V_1 \cup V_2 = V(\mathcal{H})$ such that 
    \begin{enumerate}[label=(\roman*)]
        \item\label{THM:K4C5-first-level-exact-1} $\left|\frac{|V_1|}{n} - \frac{3-\sqrt{3}}{2}\right| \le \varepsilon$, and 
        \item\label{THM:K4C5-first-level-exact-2} $\mathcal{H} \setminus \mathcal{H}[V_2] = \mathbb{B}[V_1, V_2]$. 
    \end{enumerate}
\end{theorem}

Theorem~\ref{THM:K4C5-first-level-exact}~\ref{THM:K4C5-first-level-exact-2} states that $\mathcal{H}$ is the complete semi-bipartite $3$-graph $\mathbb{B}[V_1, V_2]$ plus possibly some edges inside $V_2$. It is easy to see that $\C H$ remains $\{C_{4}^{3}, C_{5}^{3}\}$-free if we replace $\C H[V_2]$ by any $\{C_{4}^{3}, C_{5}^{3}\}$-free $3$-graph on~$V_2$. Thus, $\C H[V_2]$ is maximum $\{C_{4}^{3}, C_{5}^{3}\}$-free and Theorem~\ref{THM:K4C5-first-level-exact} applies to it provided $|V_2|\ge n_0$, and so on. We see that $\C H$ follows exactly the structure of a $\mathbb{B}_{\mathrm{rec}}$-construction, except for parts of size less than $n_0$. It follows that there is an absolute constant $C$ such that 
\begin{align}
    b_{\mathrm{rec}}(n)
    \le \mathrm{ex}(n,\{C_{4}^{3}, C_{5}^{3}\})
    \le b_{\mathrm{rec}}(n) + C,\quad\mbox{for every $n \ge 1$}. \label{eq:O(1)}
\end{align}
Thus, we know the function $\mathrm{ex}(n,\{C_{4}^{3}, C_{5}^{3}\})$ within additive $O(1)$-error.

Let us show that Theorem~\ref{THM:K4C5-first-level-exact} no longer holds if we forbid only a single tight cycle $C_{\ell}^{3}$ for any $\ell \ge 5$ with $\ell \not\equiv 0 \Mod{3}$. For this, let $\mathcal{H}$ be an $n$-vertex $\mathbb{B}_{\mathrm{rec}}$-construction with $b_{\mathrm{rec}}(n)$ edges and let $V_1 \cup V_2$ be the nontrivial partition such that $\mathcal{H}\setminus \mathcal{H}[V_2] = \mathbb{B}[V_1, V_2]$. 
First, suppose that $\ell = 5$. Then the addition of any linear\footnote{A $3$-graph is linear if every pair of edges shares at most one vertex.} $3$-graph into $V_1$ does not create a copy of $C_{5}^{3}$. 
Moreover, we can repeat this operation within $\mathcal{H}[V_2]$, and continue recursively. 
Since there exists a linear $3$-graph on $n$ vertices with $(1/6 - o(1))n^2$ edges (see e.g.~\cite{Will03}), we obtain 
\begin{align*}
    \mathrm{ex}(n,C_{5}^{3}) -  b_{\mathrm{rec}}(n)
    & \ge \sum_{i \ge 0} \frac{1}{6} \left(\left(\frac{\sqrt{3}-1}{2}\right)^{i} \cdot \frac{3-\sqrt{3}}{2} \cdot n\right)^2 -o(n^2) \\
    & = %b_{\mathrm{rec}}(n) + 
    \left(\frac{2\sqrt{3} - 3}{6} - o(1)\right) n^2.
\end{align*}
Next, let $V_{2,1} \cup V_{2,2} = V_{2}$ be the nontrivial partition such that $\mathcal{H}[V_2]\setminus \mathcal{H}[V_{2,2}] = \mathbb{B}[V_{2,1}, V_{2,2}]$ and 
fix an arbitrary vertex $v_{\ast} \in V_1$.
Suppose that $\ell \ge 7$ and $\ell \equiv 1 \Mod{3}$. Define $\mathcal{G}$ as the $3$-graph obtained from $\mathcal{H}$ by adding all triples that include $v_{\ast}$ and two vertices from $V_{2,2}$ as edges.
It is straightforward to verify that $\mathcal{G}$ is $C_{\ell}^{3}$-free and satisfies $|\mathcal{G}| = b_{\mathrm{rec}}(n) + \Omega(n^2)$. Finally, suppose that $\ell \ge 8$ and $\ell \equiv 2 \Mod{3}$. Here we can add to $\mathcal{H}$ some $3$-subsets of $V_1$ that contain $v_{\ast}$ and are disjoint otherwise. The obtained $3$-graph is $C_\ell^3$-free and has size $b_{\mathrm{rec}}(n)+\Omega(n)$. (We did not see a way to improve $b_{\mathrm{rec}}(n)$ by more than a linear function of $n$ in this case.)

\medskip\noindent\textbf{Proof outline.} As we will observe in Section~\ref{SEC:Prelim}, some standard general results on the Tur\'an densities of blowups give that it is enough to prove the versions of  Theorems~\ref{THM:turan-density-C5K4} and~\ref{THM:C5K4-stability}  when we forbid $\CC$, where
\begin{align}\label{equ:def-C'}
    %\mathrm{either}\quad 
    \CC \coloneqq \{C_{4}^{3}, C_{5}^{3}\} \quad \mathrm{or}\quad \CC \coloneqq \{C_{4}^{3}, F_1,F_2\},
\end{align}
and we define
\begin{align*}
    F_1 
    & \coloneqq \left\{\{1,2,3\}, \{1,3,4\}, \{1,4,2\}, \{1,4,5\}, \{2,3,4\}\right\} %\label{eq:F1}
    \quad\text{and}\quad 
    \\[0.3em]
    F_2 
    & \coloneqq \left\{\{1,2,3\}, \{1,3,4\}, \{1,4,5\}, \{1,5,6\}, \{1,6,2\}, \{2,3,4\}\right\}% \label{eq:F2}
\end{align*}
 to be some two non-trivial homomorphic images of $C_7^3$, see Claim~\ref{cl:hom}.
% \end{equation}

Our proofs of the corresponding results for $\CC$ crucially use the flag algebra machinery developed by Razbo\-rov~\cite{Raz07} and are computer-assisted. 
More specifically, we adopt the strategy used in the previous work~\cite{BLLP24} for determining the Tur\'{a}n density of $C_{5}^{3-}$, which is inspired by the strategy of Balogh--Hu--Lidick{\'y}--Pfender used in~\cite{BHLP16C5} for determining the inducibility of the $5$-cycle $C_5$ (where asymptotically extremal graphs are also obtained via a recursive construction). 
Our proof is based on the following two crucial claims about a $\CC$-free $3$-graph $\C H$ with $n\to\infty$ vertices and at least $\left(\alpha - o(1)\right)\binom{n}{3}$ edges, where $\alpha$ and $\gamma$ are as in~\eqref{equ:gamma-alpha-def}
\begin{enumerate}
     \item\label{it:A} Proposition~\ref{pr:2} shows that there exists a vertex partition $V_1\cup V_2 =V(\C H)$ such that $|\C H\cap \mathbb{B}[V_1,V_2]|\ge(0.408...+o(1))\binom{n}{3}$ (which is not too far from the upper bound $\binom{\gamma n}{2} (1-\gamma) n = (0.441...+o(1)) \binom{n}{3}$ coming from a maximum $\mathbb{B}_{\mathrm{rec}}$-construction).
     \item\label{it:B} Proposition~\ref{pr:3} shows that if, in addition, $\mathcal{H}$ is near-regular and the partition in Item~\ref{it:A} is assumed to be \textbf{locally maximal} (i.e. by moving any vertex between parts we cannot increase the number of edges  in $\C H\cap \mathbb{B}[V_1,V_2]$), then when comparing $\mathcal{H} \setminus \mathcal{H}[V_2]$ with the complete semi-bipartite $3$-graph $\mathbb{B}[V_1,V_2]$, the number of additional edges, up to additive $o(n^3)$, is at most $1-\frac{1}{4000}$ times the number of missing triples.
\end{enumerate}
 
Thus, we have identified a top-level partition such that, ignoring triples inside $V_2$, the $3$-graph $\C H$ does not perform better than a copy of $\mathbb{B}_{\mathrm{rec}}$ that uses the same parts. We can recursively apply this result to $\C H[V_2]$ as long as $|V_2|$ is sufficiently large. Now, routine calculations imply that $|\C H|\le (\alpha + o(1))\binom{n}{3}$. 

The main new idea of this paper (and of~\cite{BLLP24}), when compared  to~\cite{BHLP16C5,LMP24}, is a trivial combinatorial observation that for every partition $V_1\cup V_2$ of $V(\C H)$ there is also a locally maximal one with at least as many edges with exactly two vertices in $V_1$. It seems that the standard flag algebra calculations do not ``capture'' this kind of argument and doing an intermediate \textbf{local refinement} step (when  the flag algebra version  of the local maximality is manually added to the SDP program) may improve the power of the method. 

\medskip \textbf{Organisation of the paper.}
Section~\ref{SEC:Prelim} contains various definitions and auxiliary results.
The proofs of Theorems~\ref{THM:turan-density-C5K4'} and~\ref{THM:C5K4-stability'} which are the (stronger) $\CC$-versions of Theorems~\ref{THM:turan-density-C5K4} and~\ref{THM:C5K4-stability} are presented in Sections~\ref{SEC:Proof-turan-density-K4C5} and \ref{SEC:Proof-K4C5-stability} respectively. The derivation of Theorem~\ref{THM:K4C5-first-level-exact} from Theorem~\ref{THM:C5K4-stability} can be found in Section~\ref{SEC:proof-C5-exact}. Section~\ref{se:concluding} contains some concluding remarks.

%%%%%%%%%%%%%%%%%%%%%%%%%%%%%%%%%%%%%%%%%%
\section{Preliminaries}\label{SEC:Prelim}

Given a graph $G$ and two disjoint vertex subsets $S, T \subseteq V(G)$, we use $G[S,T]$ to denote the \textbf{induced bipartite subgraph} of $G$ with parts $S$ and $T$. 
For a pair of disjoint sets $V_1, V_2$, we use $K[V_1, V_2]$ to denote the complete bipartite graph with parts $V_1$ and $V_2$. 

The following two facts on graphs can be derived by straightforward greedy arguments.
\begin{fact}\label{FACT:min-deg-cleaning}
 Every $n$-vertex graph with more than $d n$ edges has  an induced (non-empty) subgraph $G'$ with minimum degree at least $d$ such that $|G'| \ge |G| - d n$.

\end{fact}

A \textbf{matching} in a graph is a collection of pairwise vertex-disjoint edges. 
\begin{fact}\label{FACT:matching-number-lower-bound}
    Every $n$-vertex graph $G$ contains a matching of size at least $\frac{|G|}{2n}$.
\end{fact}

For a set $X$ and an integer $k\ge 0$, let $\binom{X}{k}:=\{S \subseteq X \colon |S|=k\}$ denote the family of all $k$-subsets of~$X$.

Let $\C H$ be a $3$-graph. Its \textbf{order} is $v(\C H):=|V(\C H)|$.
For a vertex $v \in V(\mathcal{H})$, 
the \textbf{link} of $v$ is the following 2-graph on $V(\C H)\setminus\{v\}$:
\begin{align*}
    L_{\mathcal{H}}(v)
    \coloneqq \left\{e\in \binom{V(\mathcal{H})\setminus \{v\}}{2} \colon e \cup \{v\} \in \mathcal{H}\right\}.
\end{align*}
The \textbf{degree} $d_{\mathcal{H}}(v)$ of $v$ in $\mathcal{H}$ is given by $d_{\mathcal{H}}(v) \coloneqq |L_{\mathcal{H}}(v)|$.
We use $\delta(\mathcal{H})$, $\Delta(\mathcal{H})$, and $d(\mathcal{H})$ to denote the \textbf{minimum}, \textbf{maximum}, and \textbf{average degree} of $\mathcal{H}$, respectively.
%We will omit the subscript $\mathcal{H}$ if it is clear from the context.
For a pair of vertices $\{u,v\} \subseteq V(\mathcal{H})$, let 
\begin{align*}
    N_{\mathcal{H}}(uv)
    \coloneqq \left\{w\in V(\mathcal{H}) \colon uvw \in \mathcal{H} \right\}. 
\end{align*}
The \textbf{codegree} $d_{\mathcal{H}}(uv)$ of $\{u,v\}$ is defined by $d_{\mathcal{H}}(uv) \coloneqq |N_{\mathcal{H}}(uv)|$. 
We use $\Delta_{2}(\mathcal{H}):=\max\left\{ d_{\mathcal{H}}(uv)\colon uv\in \binom{V(\mathcal{H})}{2}\right\}$ to denote the \textbf{maximum codegree} of $\mathcal{H}$.

For an integer $m \ge 1$, the \textbf{$m$-blowup} $\C H^{(m)}$ of $\C H$ is the $3$-graph whose vertex set is the union $\bigcup_{v\in V(\C H)} U_v$ of some disjoint $m$-sets $U_v$, one per each vertex $v\in V(\C H)$, and whose edge set is the union of the complete $3$-partite $3$-graphs with parts $U_{x}, U_y, U_z$ over all edges $\{x,y,z\}\in\C H$. Informally speaking, $\C H^{(m)}$ is obtained from $\C H$ by cloning each vertex $m$ times.

\begin{claim}\label{cl:hom} For every $\ell\ge 8$ not divisible by $3$, the tight $\ell$-cycle is a subgraph of some blowup of $C_4^3$ and of $C_5^3$. Also, $C_{7}^{3}$ appears in a blowup of each of $C_4^3$, $F_1$ and $F_2$, where $F_1$ and $F_2$ are as in~\eqref{equ:def-C'}.% 
\end{claim}

\begin{proof}[Proof of Claim~\ref{cl:hom}] 
To show that $C_\ell^3$ appears in a blowup of $F$, we have to give a \textbf{homomorphism} $\phi$ from $C_\ell^3$ to $F$, that is, a (not necessarily injective) map from $V(C_\ell^3)=[\ell]$ to $V(F)$ that sends edges to edges.

Note that for each $m\ge 4$, $C_{m+3}^3$ admits a homomorphism into $C_{m}^{3}$: e.g., take the identity map on $[m]\subseteq V(C_{m+3}^3)$ and map the remaining vertices $m+1,m+2,m+3$ to respectively $1,2,m$. Thus, $C_\ell^3$ admits a homomorphism to every tight $\ell'$-cycle with $\ell'\le \ell$ being congruent to $\ell$ modulo 3. For any $\ell\ge 8$ not divisible by $3$, the set of such lengths $\ell'$ is easily seen to contain both $4i$ and $5j$ for some integers $i,j\ge 1$. In turn, $C_{4i}^3$ and $C_{5j}^3$ admit homomorphisms to $C_4^3$ and $C_5^3$ respectively (just by winding around the smaller cycle), proving the first part. 

We can obtain $F_2$ from $C_7^3$ by contracting an uncovered pair (which is unique up to an automorphism of~$C_7^3$). Thus, for a homomorphism, we can take the identity map on $[4]\subseteq V(C_7^3)$ and send the remaining vertices $5,6,7$ to respectively $1,5,6$. Furthermore, a homomorphism from $F_2$ to $F_1$ can be obtained by taking the identity map on $[5]\subseteq V(F_2)$ and sending $6$ to~$4$. The second part now follows by combining some of the above observations. \end{proof}

Suppose that $\mathcal{F}'$ is obtained from some forbidden $3$-graph family $\C F$ by replacing a member $\mathcal{H} \in \mathcal{F}$ with some set of $3$-graphs such that $\mathcal{H}$ appears in a blowup of each. Let $\mathcal{G}$ be any $\C F$-free $3$-graph of order $n\to\infty$.
The number of (non-induced) copies of each $F\in \C F'\setminus \C F$ in $\mathcal{G}$ is $o(n^{v(F)})$ by e.g.,\ the result of Erd\H os~\cite{Erd64KST}, since the $\mathcal{H}$-free $3$-graph $\mathcal{G}$ cannot have large blowups of~$F$. Thus, by the Hypergraph Removal Lemma (see, e.g.,~\cite{RS04,NRS06,Gow07}), $\mathcal{G}$ can be made $\C F'$-free by removing $o(n^3)$ edges. It follows that $\pi(\C F)\le \pi(\C F')$. Moreover, if every $\mathbb{B}_{\mathrm{rec}}$-construction is $\C F$-free and we can prove the Erd{\H o}s--Simonovits-type stability of Theorem~\ref{THM:C5K4-stability}  for $\C F'$ 
%as the forbidden family 
then it also holds for $\C F$. Thus, Theorems~\ref{THM:turan-density-C5K4} and~\ref{THM:C5K4-stability} follow from the following two results by Claim~\ref{cl:hom}.
%(Recall that $\alpha:=2\sqrt{3}-3$.)
% Consequently, we have 
% \begin{align}\label{equ:C8-C4C5-hom}
%     \pi(C_{\ell}^{3}) 
%     \le \pi(\{C_{4}^{3}, C_{5}^{3}\}). 
% \end{align}

\begin{theorem}\label{THM:turan-density-C5K4'} 
    Let $\CC$ be as in~\eqref{equ:def-C'}. Then it holds that $\pi(\CC) \le \alpha$.
\end{theorem}

\begin{theorem}\label{THM:C5K4-stability'}
 Let $\CC$ be as in~\eqref{equ:def-C'}.  Then for every $\varepsilon > 0$, there exist $\delta > 0$ and $n_0$ such that every $\CC$-free $3$-graph with $n\ge n_0$ vertices and  at least $\alpha \binom{n}{3} - \delta n^3$ edges is a $\mathbb{B}_{\mathrm{rec}}$-subconstruction after removing at most $\varepsilon n^3$ edges.      
\end{theorem}

Following the definition in~\cite{LMR23unified}, a family $\mathcal{F}$ of $r$-graphs is \textbf{blowup-invariant} if every blowup of an $\mathcal{F}$-free $r$-graph remains $\mathcal{F}$-free. It is routine to check that the family $\mathcal{C}'$ is blowup-invariant.
We will use the following two properties of blowup-invariant families, referred to as \textbf{boundedness} and \textbf{smoothness} in~\cite{HLLYZ23}. 
\begin{proposition}\label{PROP:blowup-invariant-bounded}
    Let $r\ge 2$. Suppose that a family $\mathcal{F}$ of $r$-graphs is blowup-invariant. Then the following statements hold. 
    \begin{enumerate}[label=(\roman*)]
        \item\label{PROP:blowup-invariant-bounded-1} 
        %There exists an absolute constant $\delta_{\ref{PROP:blowup-invariant-bounded}} = \delta_{\ref{PROP:blowup-invariant-bounded}}(r) > 0$ such that 
        For every $\delta \in \left(0, 1/(2r)^r\right)$, there exists $n_0 = n_0(\delta, \mathcal{F})$ such that every  $\mathcal{F}$-free $r$-graph $\mathcal{H}$ with $n\ge n_0$ vertices and at least $\left(\frac{\pi(\mathcal{F})}{r!} - \delta\right) n^{r}$ edges satisfies $\Delta(\mathcal{H}) \le \left( \frac{\pi(\mathcal{F})}{(r-1)!} + 10 \delta^{1/r} \right) n^{r-1}$.
        \item\label{PROP:blowup-invariant-bounded-2} Suppose that $\mathcal{H}$ is an $\mathcal{F}$-free $r$-graph on $n$ vertices with exactly  $\mathrm{ex}(n,\mathcal{F})$ edges. 
        Then $\Delta(\mathcal{H}) - \delta(\mathcal{H}) \le \binom{n-2}{r-2}$.
    \end{enumerate}
\end{proposition}
\begin{proof}[Proof of Proposition~\ref{PROP:blowup-invariant-bounded}]
    First, we prove Proposition~\ref{PROP:blowup-invariant-bounded}~\ref{PROP:blowup-invariant-bounded-1}. 
    Fix $r\ge 2$ and $\delta \in \left(0, 1/(2r)^r\right)$, noting that 
    \begin{align}\label{equ:PROP:blowup-invariant-bounded-a}
        2\binom{r}{2}\frac{\pi(\mathcal{F})}{r!}
        \le 2\binom{r}{2}\frac{1}{r!} 
        \le 1
        \quad\text{and}\quad 
        1 + r \delta^{1/r} + 2\binom{r}{2} \delta^{2/r} 
        \le 2. 
    \end{align}
    Let $\varepsilon \coloneqq 10 \delta^{1/r}$ and $n$ be sufficiently large. 
    Suppose to the contrary that there exists an $\mathcal{F}$-free $r$-graph $\mathcal{H}$ on $n$ vertices with at least $\left(\frac{\pi(\mathcal{F})}{r!} - \delta\right) n^{r}$ edges, but $\Delta(\mathcal{H}) > \left( \frac{\pi(\mathcal{F})}{(r-1)!} + \varepsilon \right) n^{r-1}$.
    Let $v_{\ast} \in V(\mathcal{H})$ be a vertex of maximum degree. 
    Let $\mathcal{G}$ be the $r$-graph obtained from $\mathcal{H}$ by duplicating $v_{\ast}$ for $\delta^{1/r} n$ times. 
    Note that $\mathcal{G}$ is an $r$-graph on $(1+\delta^{1/r}) n$ vertices. 
    Since $\mathcal{F}$ is blowup-invariant, $\mathcal{G}$ remains $\mathcal{F}$-free.
    Also, it follows from the definition of $\mathcal{G}$ that 
    \begin{align}\label{equ:PROP:blowup-invariant-bounded-G-lower-bound}
        |\mathcal{G}|
        = |\mathcal{H}| + \delta^{1/r} n \cdot \Delta(\mathcal{H}) 
        & > \left(\frac{\pi(\mathcal{F})}{r!} - \delta\right) n^{r} + \delta^{1/r} \left( \frac{\pi(\mathcal{F})}{(r-1)!} + \varepsilon \right) n^{r} \notag \\ 
        & = \left(1+r\delta^{1/r}\right)\frac{\pi(\mathcal{F})}{r!}\, n^{r} + \left(\delta^{1/r} \varepsilon - \delta\right) n^r \notag \\
        & \ge \left(1+r\delta^{1/r}\right)\frac{\pi(\mathcal{F})}{r!}\, n^{r} + 9 \delta^{2/r} n^r, 
    \end{align}
    where in the last inequality, we used the assumption that $\varepsilon = 10 \delta^{1/r}$. 
    
    On the other hand, since $n$ is sufficiently large, we have 
    \begin{align*}
        |\mathcal{G}|
        \le \mathrm{ex}(n+\delta^{1/r}n, \mathcal{F}) 
        \le \left(\frac{\pi(\mathcal{F})}{r!} + \delta\right) \left(1 + \delta^{1/r}\right)^{r} n^{r}.
    \end{align*}
    Using~\eqref{equ:PROP:blowup-invariant-bounded-a} and the inequality $(1+x)^{r} \le 1+rx + 2\binom{r}{2}x^2$ for $x\in [0,1/2^r]$, we obtain  
    \begin{align*}
        |\mathcal{G}| 
        & \le \left(\frac{\pi(\mathcal{F})}{r!} + \delta\right) \left(1 + r \delta^{1/r} + 2\binom{r}{2} \delta^{2/r}\right) n^{r} \\
        & = \left(1 + r \delta^{1/r} \right)\frac{\pi(\mathcal{F})}{r!} n^{r} + 2\binom{r}{2} \delta^{2/r} \frac{\pi(\mathcal{F})}{r!} n^{r} + \left(1 + r \delta^{1/r} + 2\binom{r}{2} \delta^{2/r}\right) \delta n^{r} \\
        & \le \left(1 + r \delta^{1/r} \right)\frac{\pi(\mathcal{F})}{r!} n^{r} + \delta^{2/r} n^{r} + 2 \delta n^{r}
        < \left(1 + r \delta^{1/r} \right)\frac{\pi(\mathcal{F})}{r!} n^{r} + 3\delta^{2/r} n^{r}, 
    \end{align*}
    which is a contradiction to~\eqref{equ:PROP:blowup-invariant-bounded-G-lower-bound}. 

    Now we prove Proposition~\ref{PROP:blowup-invariant-bounded}~\ref{PROP:blowup-invariant-bounded-2}. 
    Fix $n \in \mathbb{N}$. Suppose to the contrary that there exists an $\mathcal{F}$-free $r$-graph $\mathcal{H}$ on $n$ vertices with exactly $\mathrm{ex}(n,\mathcal{F})$ edges, but $\Delta(\mathcal{H}) - \delta(\mathcal{H}) > \binom{n-2}{r-2}$.
    Let $v_{\ast} \in V(\mathcal{H})$ be a vertex of maximum degree and let $u_{\ast} \in V(\mathcal{H})$ be a vertex of minimum degree. 
    Let $\mathcal{G}$ be the $r$-graph obtained from $\mathcal{H}$ by first removing the vertex $u_{\ast}$ and then duplicating $v_{\ast}$ once. 
    Since $\mathcal{F}$ is blowup-invariant, $\mathcal{G}$ remains $\mathcal{F}$-free. 
    However, 
    \begin{align*}
        |\mathcal{G}|
        & = |\mathcal{H}| - d_{\mathcal{H}}(u_{\ast}) + d_{\mathcal{H}}(v_{\ast}) - d_{\mathcal{H}}(u_{\ast} v_{\ast}) \\
        & \ge \mathrm{ex}(n,\mathcal{F}) - \delta(\mathcal{H}) + \Delta(\mathcal{H}) - \binom{n-2}{r-2} 
        > \mathrm{ex}(n,\mathcal{F}), 
    \end{align*}
    a contradiction. 
\end{proof}

It is easy to see that the family $\{C_{4}^{3}, C_{5}^{3}\}$ is blowup-invariant, since the shadows of both $C_{4}^{3}$ and $C_{5}^{3}$ are  complete. 
It is also straightforward to verify that the family $\{C_{4}^{3}, F_1, F_2\}$ is blowup-invariant.
Therefore, Proposition~\ref{PROP:blowup-invariant-bounded} immediately yields the following corollary.
% We will first establish an upper bound for the maximum degree of a nearly extremal $\{C_{4}^{3}, C_{5}^{3}\}$-free $3$-graph and a lower bound for the minimum degree of an extremal $\{C_{4}^{3}, C_{5}^{3}\}$-free $3$-graph. 
%
\begin{proposition}\label{PROP:C5minus-max-degree}  
    For every $\delta \in \left(0, 1/6^3\right)$ there exists $n_0$ such that the following statements hold for all $n \ge n_0$. 
    \begin{enumerate}[label=(\roman*)]
        \item\label{PROP:C5minus-max-degree-1} 
        Suppose that $\mathcal{H}$ is a $\CC$-free $3$-graph on $n$ vertices with at least $\left(\frac{\pi(\CC)}{6} - \delta\right)n^3$ edges. 
        Then $\Delta(\mathcal{H}) \le \left(\frac{\pi(\CC)}{2} + 10\delta^{1/3} \right)n^2$.
        \item\label{PROP:C5minus-max-degree-2} Suppose that $\mathcal{H}$ is a $\CC$-free $3$-graph on $n$ vertices with exactly $\mathrm{ex}(n,\CC)$ edges. 
        Then $\delta(\mathcal{H}) \ge \left(\frac{\pi(\CC)}{2} - \delta \right)n^2$.\qed
    \end{enumerate}
\end{proposition}

The following fact about small-degree vertices in near-extremal $\mathcal{F}$-free hyprgraphs can be proved using a simple deletion argument and has appeared in various places (see e.g.~{\cite[Lemma~4.2]{LMR1}} and~{\cite[Fact~2.5]{LMR23unified}}). 
For completeness, we include its proof here. 

\begin{fact}\label{FACT:near-extremal-small-deg}
    Let $\mathcal{F}$ be a family of $r$-graphs.
    For every $\delta > 0$ there exists $N_0$ such that the following holds for all $n \ge N_0$. 
    Suppose that $\mathcal{H}$ is an $\mathcal{F}$-free $r$-graph on $n$ vertices with $|\mathcal{H}| \ge \left(\pi(\mathcal{F})/r! - \delta \right)n^{r}$.
    Then the set 
    \begin{align}\label{equ:Z-delta-H-def}
        Z_{\delta}(\mathcal{H})
        \coloneqq \left\{v\in V(\mathcal{H}) \colon d_{\mathcal{H}}(v) \le \left(\pi(\mathcal{F})/(r-1)! - 4 \delta^{1/2} \right)n^{r-1}\right\} 
    \end{align}
    has size at most $\delta^{1/2} n$. 
\end{fact}
\begin{proof}[Proof of Fact~\ref{FACT:near-extremal-small-deg}]
    Fix $\delta > 0$. 
    Let $\varepsilon \coloneqq \delta^{1/2}$, noting that $- 4 \varepsilon \delta^{1/2} + \varepsilon^{2} + \delta < -\delta$. 
    Choose $n$ to be sufficiently large so that, in particular, 
    \begin{align*}
        \mathrm{ex}(m, \mathcal{F})
        \le \left(\frac{\pi(\mathcal{F})}{r!} + \delta \right)m^{r}
        \quad\text{for all}\quad m \ge \frac{n}{2}. 
    \end{align*}
    Suppose to the contrary that there exists an $n$-vertex $\mathcal{F}$-free $r$-graph $\mathcal{H}$ with $|\mathcal{H}| \ge \left(\pi(\mathcal{F})/r! - \delta \right)n^{r}$ and $|Z_{\delta}(\mathcal{H})| > \varepsilon n$.
    Let $Z \subseteq Z_{\delta}(\mathcal{H})$ be a subset of size $\varepsilon n$.
    Notice that 
    \begin{align*}
        |Z|  \left(\frac{\pi(\mathcal{F})}{(r-1)!} - 4 \delta^{1/2} \right)n^{r-1}
        & = \varepsilon n  \left(\frac{\pi(\mathcal{F})}{(r-1)!} - 4 \delta^{1/2} \right)n^{r-1} 
        =  \left(\frac{\pi(\mathcal{F})}{(r-1)!} \varepsilon - 4 \varepsilon \delta^{1/2} \right) n^{r}.
    \end{align*}
    Let $U \coloneqq V(\mathcal{H}) \setminus Z$ and let $\mathcal{G}$ denote the induced subgraph of $\mathcal{H}$ on $U$.  
    Using the inequalities $\frac{\pi(\mathcal{F})}{r!} \binom{r}{2} \le \frac{1}{r!} \binom{r}{2} \le 1$ and $(1-x)^r \le 1-rx+\binom{r}{2}x^2$ for $x\in [0,1]$, we obtain 
    \begin{align*}
        |\mathcal{G}|
        \le \mathrm{ex}\left(n - \varepsilon n, \mathcal{F} \right)  
        \le \left(\frac{\pi(\mathcal{F})}{r!} + \delta \right) \left(1-\varepsilon \right)^{r} n^{r} 
        & \le \frac{\pi(\mathcal{F})}{r!} \left(1- \varepsilon \right)^{r} n^{r} + \delta n^r \\
        & \le \left(\frac{\pi(\mathcal{F})}{r!} -  \frac{\pi(\mathcal{F})}{(r-1)!} \varepsilon +  \varepsilon^{2}\right) n^{r} + \delta n^{r}.
    \end{align*}
    It follows from the definition of $Z_{\delta}(\mathcal{H})$ that 
    \begin{align*}
        |\mathcal{H}|
        & \le |\mathcal{G}| + |Z|  \left(\frac{\pi(\mathcal{F})}{(r-1)!} - 4\delta^{1/2} \right)n^{r-1}  \\
        & \le \left(\frac{\pi(\mathcal{F})}{r!} -  \frac{\pi(\mathcal{F})}{(r-1)!} \varepsilon +  \varepsilon^{2}\right) n^{r} + \delta n^{r} 
                + \left(\frac{\pi(\mathcal{F})}{(r-1)!} \varepsilon - 4 \varepsilon \delta^{1/2} \right) n^{r}  \\ 
        & = \left(\frac{\pi(\mathcal{F})}{r!} - 4 \varepsilon \delta^{1/2} + \varepsilon^{2} + \delta \right)n^{r}
        < \left(\frac{\pi(\mathcal{F})}{r!} - \delta \right)n^{r}, 
    \end{align*}
    a contradiction to the assumption that $|\mathcal{H}| \ge \left(\pi(\mathcal{F})/r! - \delta \right)n^{r}$. 
\end{proof}

We will use the global constants $\alpha$ and $\gamma$ defined in~\eqref{equ:gamma-alpha-def}. Recall that they are respectively the upper bound on the Tur\'an density that we want to prove and the asymptotically optimal relative size for $V_1$ in a large $\mathbb{B}_{\mathrm{rec}}$-construction.
The \textbf{standard $1$-dimensional simplex} is
\begin{align*}
    \mathbb{S}^2
    \coloneqq \left\{(x_1, x_2) \in \mathbb{R}^2 \colon x_1+x_2 = 1 \text{ and } x_i \ge 0 \text{ for } i \in [2]\right\}.
\end{align*}
The following fact is straightforward to verify. 
\begin{fact}\label{FACT:ineqality}
    The following inequalities hold for every $(x_1, x_2) \in \mathbb{S}^{2}$ with $x_2 < 1 \colon$  
    \begin{enumerate}[label=(\roman*)]
        \item\label{FACT:ineqality-1} $\frac{x_1^2 x_2}{2\left(1 - x_2^3\right)} \le \frac{\alpha}{6}$. 
        \item\label{FACT:ineqality-2} Suppose that $x_1 \in \left[\frac{1}{2}, 1\right]$. Then  
        \begin{align*}
            \frac12\, x_1^2 x_2 + \frac{\alpha}{6}\, x_2^3 
            \le \frac{\alpha}{6} - \frac{1}{4} \left(x_1 - \gamma\right)^2.
        \end{align*}
    \end{enumerate}
\end{fact}

%%%%%%%%%%%%%%%%%%%%%%%%%%%%%%%%%%%%%%%%%%%%%%%%%%%%%
\section{Computer-generated results}\label{se:Flag-Algebra}
In this section, we present the results derived by computer using the flag algebra method of Razborov~\cite{Raz07}, which is also described in e.g.~\cite{Razborov10,BaberTalbot11,SFS16,GilboaGlebovHefetzLinialMorgenstein22}.
Since this method is well-known by now, we will be very brief. In particular,  we omit many definitions, referring the reader to~\cite{Raz07,Razborov10} for any missing notions. Roughly speaking, a \textbf{flag algebra proof using $0$-flags on $m$ vertices} of an upper bound $u\in\I R$ on the given objective function $f$ 
%that can be written as a fixed linear combination of densities 
consists of an identity 
$$
u-f(\C H)=\mathrm{SOS}+\sum_{F\in\C F_m^0}c_F \cdot p(F,\C H)+o(1),
$$ 
which is asymptotically true for any admissible $\C H$ with $|V(\C H)|\to\infty$. 
Here the $\mathrm{SOS}$-term can be represented as a sum of squares (as described e.g.\ in~\ra{Section~3}), $\C F_m^0$ consists of all up to isomorphism \textbf{$0$-flags} (objects of the theory without designated roots) with $m$ vertices,  each coefficient $c_F\in\I R$ is non-negative, and 
$p(F,\C H)\in [0,1]$ is the density of $k$-subsets of $V(\C H)$ that span a subgraph isomorphic to $F$. If $f(\C H)$ can be represented as a linear combination of the densities of members of $\C F_m^0$ in $\C H$ then finding the smallest possible $u$ amounts to solving a semi-definite program (SDP) with $|\C F_m^0|$ linear constraints. 
%(So we write the size of $\C F_m^0$ in each case to give the reader some idea of the size of the programs that we had to solve.)

We formed the corresponding SDPs and then analysed the solutions returned by computer, using a modified version of the SageMath package. This package is still under development, a short guide on how to install it and an overview of its current functionality can be found in the GitHub repository \href{https://github.com/bodnalev/sage}{\url{https://github.com/bodnalev/sage}}. 
The scripts that we used to generate the certificates and the certificates themselves can be found in the ancillary folder of the arXiv version of this paper or in a separate GitHub repository \href{https://github.com/bodnalev/supplementary_files}{\url{https://github.com/bodnalev/supplementary_files}} in the folder \verb$no_cl$. This folder also includes a self-contained verifier entirely written in Python 3, with a hard-coded version of the objective functions and assumptions found in the propositions of this section. The purpose of this program is to first ensure that the objective functions and assumptions match those in the propositions, and then to verify the correctness of the flag algebra proofs from the certificates. %That is, to confirm that the set of inequalities corresponding to the base 0-flags are all (or some, at the user's request) satisfied. 
We did not optimise the verifier for speed (aiming instead for the clarity of the code); the full verification of all computer-generated results in this paper takes about 80 hours on an average PC.

As far as we can see, none of the certificates  can be made sufficiently compact to be human-checkable. 
Hence we did not make any systematic efforts to reduce the size of the obtained certificates, being content with ones that can be generated and verified on an average modern PC. In particular, we did not try to reduce the set of the used types needed for the proofs, although we did use the (standard) observation of Razborov~\ra{Section 4} that each unknown SDP matrix can be assumed to consist of 2 blocks (namely, the invariant and anti-invariant parts).

\newcommand{\al}[1]{\alpha_{\mathrm{\ref{pr:#1}}}}
\newcommand{\be}[1]{\beta_{\mathrm{\ref{pr:#1}}}}
\newcommand{\ep}[1]{\e_{\mathrm{\ref{pr:#1}}}}

\newcommand{\AlphaThreeTwoRat}{\frac{11798651975724865057843}{28855039009153351680000}}
\newcommand{\AlphaThreeTwoAppr}{0.40889...} %0.408893988047707

Our first result establishes a crude upper bound on  the Tur\'{a}n density of $\mathcal{C}'$, namely that $\pi(\mathcal{C'})\le \al{1}$, where
\begin{align*}
    \al{1}
    \coloneqq 
    \begin{cases}
        \frac{7309337}{15728640} \approx 0.46471... & \quad\text{if}\quad \mathcal{C}' = \{C_{4}^{3}, C_{5}^{3}\}, \\[0.5em]
        \frac{60891}{131072} \approx 0.46456... & \quad\text{if}\quad \mathcal{C}' = \{C_{4}^{3}, F_1, F_2\},
    \end{cases}
\end{align*}
which is not far from the conjectured value $2\sqrt{3}-3 \approx 0.46410$.

\begin{proposition}\label{PROP:crude-upper-bound-pi-C5}\label{pr:1}
    For every integer $n \ge 1$, every $\mathcal{C}'$-free $n$-vertex $3$-graph $\mathcal{H}$ has at most $\al{1} \frac{n^3}{6}$ edges. 
\end{proposition}
\begin{proof}[Proof of Proposition~\ref{PROP:crude-upper-bound-pi-C5}]
    Suppose on the contrary that some $3$-graph $\C H$  contradicts the proposition. Thus, $\beta \coloneqq 6|\C H|/(v(\C H))^3$ is greater than $\al1$. For every integer $k\ge 1$, the $k$-blowup $\C H^{(k)}$ has $k\, v(\C H)$ vertices and $k^3\,|\C H|$ edges, so the ratio $6|\C H^{(k)}|/(v(\C H^{(k)}))^3$ for $\C H^{(k)}$ is also equal to~$\beta$. Also, $\C H^{(k)}$ is still $\mathcal{C}'$-free since this family can be easily checked to be closed under taking images under homomorphisms.

    The sequence $(\C H^{(k)})_{k=1}^\infty$ converges as $k\to\infty$ to a limit $\phi$, where 
     $$
     \phi(F):=\lim_{k\to\infty} p(F,\C H^{(k)}),\quad\mbox{for a 3-graph $F$,}
     $$
     that is, $\phi$
    sends a 3-graph $F$ to the limiting density of $F$ in $\C H^{(k)}$. We extend $\phi$ by linearity to formal linear combinations of $3$-graphs, obtaining a positive  homomorphism $\C A^0\to \I R$ from the flag algebra $\C A^0$ to the reals, see~\cite[Section~3.1]{Raz07}. It satisfies that $\phi(K_3^3)=\beta$, that is, the edge density in the limit $\phi$ is $\beta$. 
    
    However, our (standard) application of the flag algebra method using $\mathcal{C}'$-free $3$-graphs on $6$ vertices  gives $\al1$ as an upper bound on the edge density. This contradicts our assumption $\beta>\al1$. 
\end{proof}%PROP

For an $n$-vertex $3$-graph $\C H$ define the \textbf{max-cut ratio} to be
\begin{align*}
    \mu(\mathcal{H})
    \coloneqq \frac{6}{n^3} \cdot \max\left\{|\mathcal{H} \cap \mathbb{B}[V_1, V_2]| \colon \text{$V_1, V_2$ form a partition of $V(\mathcal{H})$}\right\}. 
\end{align*}

The following result shows that an arbitrary $\CC$-free $3$-graph $\C H$ of fairly large size contains a large semi-bipartite subgraph, namely with at least  $\al2 n^3/6$ edges, where 
\begin{equation}\label{eq:al2}
    \al2 
    \coloneqq 
    \begin{cases}
        \frac{1194242541}{2923734800} \approx 0.40846... & \quad\text{if}\quad \mathcal{C}' = \{C_{4}^{3}, C_{5}^{3}\}, \\[0.5em]
        \frac{2013187661}{4871280000} \approx 0.41327... & \quad\text{if}\quad \mathcal{C}' = \{C_{4}^{3}, F_1, F_2\}.
    \end{cases}
    % \coloneqq \AlphaThreeTwoRat \approx \AlphaThreeTwoAppr. 
\end{equation}

\begin{proposition}\label{pr:Flag-3-parts}\label{pr:2} 
    Suppose that $\mathcal{H}$ is an $n$-vertex $\CC$-free $3$-graph with at least $\be2\,\frac{n^3}{6}$ edges, where 
    \begin{align*}
        \be2 
        \coloneqq \frac{4641}{10000} \le 2\sqrt{3} - 3 - 10^{-6}. 
    \end{align*}
    Then $\mathcal{H}$ has the max-cut ratio $\mu(\C H)$ at least $\al2$.
\end{proposition}
\begin{proof}[Proof of Proposition~\ref{pr:Flag-3-parts}]
Informally speaking, the main idea is that any edge $X$ of  $\C H$ defines two disjoint subsets  $V_1^{X}, V_2^{X}\subseteq V(\C H)\setminus X$, where 
\begin{enumerate}[label=(\roman*)]
    \item\label{it:V1-X} $V_1^{X}$ consists of vertices in $V(\C H)\setminus X$ that have exactly two links in $X$, that is, we put a vertex $v$ into $V_1$ if there are exactly two pairs in $X$ that form an edge with $v$;  
    \item\label{it:V2-X} $V_2^{X}$ consists of vertices in $V(\C H)\setminus X$ that have either zero or exactly one link in $X$. 
    % \item\label{it:T-X} and $T^{X} \coloneqq V(\C H)\setminus (X\cup V_1 \cup V_2)$. 
\end{enumerate}
%
%In the case when $\mathcal{C'}$ contains $C_{4}^{3}$, the sets $V_1^{X}$ and  $V_2^{X}$ form a partition of $V(\C H)\setminus X$. In any case, 
Of course, $(6/n^3)\,|\mathcal{H} \cap \mathbb{B}[V_1, V_2]|$ is a lower bound on the max-cut ratio of $\mathcal{H}$.
%Note from the $C_{4}^{3}$-freeness that there is no vertex in $V(\C H)\setminus X$ that has three links in $X$. Thus, $V_1^{X} \cup V_2^{X}$ is indeed a partition of $V(\C H)\setminus X$. 
%\lb{Shouldn't this simplify the above to an empty T?}
% We ignore the vertices in $T^{X}$ (although we could have, for instance, assigned them randomly into parts $V_{1}^{X}$ and $V_{2}^{X}$ to obtain a slightly better bound, the stated bound suffices for our purposes) and consider only the two disjoint subsets $V_1^X,V_2^X \subseteq V(\C H)$.
We choose an edge $X\in\C H$ such that the size of $\C G^X \coloneqq \C H\cap \mathbb{B}[V_1^X,V_2^X]$ is at least the average value when $X$ is chosen uniformly at random from~$\C H$. 

We can express the product $P:=p(K_3^3,\C H)\cdot \I E|\C G^X|/{n-3\choose 3}$ via densities of $6$-vertex subgraphs, where $p(K_3^3,\C H)=|\C H|/\binom{v(\C H)}{3}$ is the \textbf{edge density} of~$\C H$.
Indeed, $P$ is the probability, over random disjoint $3$-subsets $X,Y\subseteq V(\C H)$, of $X\in\C H$ and $Y\in\C G^X$; this in turn can be determined by first sampling a random 6-subset $X\cup Y\subseteq V(\C H)$ and computing its contribution to $P$ (which depends only on the isomorphism class of $\C H[X\cup Y]$). 
%Thus, as a lower bound on $\max_{X\in\C H} |\C G^X|/{n-3\choose 3}$, we can take the ratio of the minimum value of $P$ to the maximum possible edge density (which was already  upper bounded by Proposition~\ref{pr:Flag-raw-upper-bound}). 
We bound $P$ from below,  via rather standard flag algebra calculations.

Let us briefly give some formal details. We work with the limit theory of $\CC$-free $3$-graphs. For a $k$-vertex \textbf{type} (i.e.,\ a fully labelled $3$-graph) $\tau$ and an integer $m\ge k$, let $\C F_m^\tau$ be the set of all \textbf{$\tau$-flags} on $m$ vertices (i.e.,\ $3$-graphs with $k$ labelled vertices that span a copy of $\tau$) up to label-preserving isomorphisms. Let $\I R\C F_m^\tau$ consist of  formal linear combinations $\sum_{F\in \C F_m^\tau} c_F F$ of $\tau$-flags with real coefficients (which we will call \textbf{quantum $\tau$-flags}).

For $i\in \{0,1,2\}$, the unique type with $i$ vertices (and no edges) is denoted by $i$. Let $E$ denote the type which consists of three roots spanning an edge.

We will use the following definitions, depending on an $E$-flag $(H, (x_1,x_2,x_3))$. (Thus, $\{x_1,x_2,x_3\}$ is an edge of a $\CC$-free $3$-graph $H$.) For $i\in[3]$, we define $V_1=V_1(H, (x_1,x_2,x_3))$ to consist of those vertices $y\in V(H)\setminus X$ such that the $H$-link of $y$ contains exactly two pairs from $X$, and $V_2=V_2(H, (x_1,x_2,x_3))$ to consist of those vertices $z\in V(H)\setminus X$ such that the $H$-link of $z$ contains either zero or exactly one pair from $X$, where we denote $X\coloneqq \{x_1,x_2,x_3\}$. 
By definition, the sets $V_1$ and $V_2$ are pairwise disjoint. 
Let 
\begin{align*}
    \mathcal{B} 
    = \mathcal{B}(H, (x_1,x_2,x_3)) 
    \coloneqq H\cap \mathbb{B}[V_1,V_2]
\end{align*}
be the $3$-graph consisting of those edges in $H$ that contains two vertices from $V_1$ and one vertex from $V_2$.

When we normalise $|\mathcal{B}|$ by ${v(H)-3\choose 3}^{-1}$, the obtained ratio can be viewed as the probability over a uniformly random $3$-subset $Y$ of $V(H)\setminus\{x_1,x_2,x_3\}$  that $Y$ is an edge of $H$, the link of two vertices of $Y$ have exactly two pairs inside $X$ and the link of the remaining vertex of $Y$ has either zero or exactly one pair inside $X$. 
This ratio is the \textbf{density}  (where the roots have to be preserved) of the quantum $E$-flag
% the following linear combination
\begin{equation}\label{eq:F222E}
	Y^E\coloneqq \sum_{F\in \C F_6^E} |\mathcal{B}(F)|\,  F
	%\in\I R\C F_6^E
\end{equation}
in the $E$-flag $(H,(x_1,x_2,x_3))$.
Note that each coefficient $|\mathcal{B}(F)|$ in~\eqref{eq:F222E} is either 0 or 1 since there is only one potential set $Y$ to test inside every $F\in\C F_6^E$ (while the scaling factor ${6-3\choose 3}^{-1}=1$ is omitted from~\eqref{eq:F222E}).

For a $k$-vertex type $\tau$ and $(F,(x_1,\dots,x_k))\in\C F_m^\tau$,  the \textbf{averaging} $\eval{F}$ is defined as the quantum (unlabelled) $0$-flag $q\, F\in\I R\C F_m^0$, where $q$ is the probability for a uniform random injection $f:[k]\to V(F)$ that $(F,(f(1),\dots,f(k)))$ is isomorphic to $(F,(x_1,\dots,x_k))$, and this definition is extended to $\I R\C F_m^\tau$ by linearity, see e.g.\ \cite[Section 2.2]{Raz07}.

Let us return to the proposition. Suppose that some $n$-vertex $3$-graph $\C H$ contradicts it. Let $\beta:=\mu(\C H)$ be the max-cut ratio for $\C H$.
By our assumption, $\beta<\al2$.

Let $\phi:\C A^0\to \I R$ be the limit of the uniform blowups $\C H^{(k)}$ of $\C H$. The assumption on the edge density of $\C H$ gives that $\phi(K_3^3)\ge \be2$. 

Note that $\mu(\C H^{(k)})=\mu(\C H)$, where the non-trivial inequality $\mu(\C H^{(k)})\le \mu(\C H)$ can be proved as follows. Take an optimal partition $V_1\cup V_2$ of $V(\C H^{(k)})$. If some part $U_x$ for $x\in V(\C H)$ is properly split by the partition  then, since the number of $\I B[V_1,V_2]$-edges is a linear function of $|U_x\cap V_1|$, so we can re-assign all vertices of $U_x$ to one part with the new partition still optimal. We iteratively repeat this step for all~$x\in V(\C H)$. The final partition of $V(\C H^{(k)})$ does not split any $U_x$ and thus gives a partition of $V(\C H)$ of the same cut value.

For each $k$-blowup $\C H^{(k)}$ pick an edge $X = \{x_1,x_2,x_3\}\in \C H^{(k)}$ such that the density (when normalised by ${kn-3\choose 3}^{-1}$) of $\mathcal{H}^{(k)} \cap \mathbb{B}[V_{1}^{X}, V_{2}^{X}]$ is at least the average value. This average density tends to $\phi(\eval{Y^E})/\phi(K_3^3)$ as $k\to\infty$. 
Proposition~\ref{PROP:crude-upper-bound-pi-C5} gives that $\phi(K_3^3)\le \al{1}$. Thus, in order to to get a contradiction to  $\mu(\C H^{(k)})=\mu(H)=\beta<\al2$, it is enough to prove that $\phi(\eval{Y^E}) \ge \al2\al1$.

This is done by flag algebra calculations on flags with at most $6$ vertices, under the assumption that $\phi(K_3^3)\ge \be2$.  (In fact, our definition of $\al2$ comes from taking the rounded lower bound on $\phi(\eval{Y^E})$ returned by the computer and dividing it by $\al1$.)
\end{proof}

In order to state the next result, we have to provide various definitions for a $3$-graph~$\C H$. Recall that a partition $V(\C H)=V_1\cup V_2$ of its vertex set is \textbf{locally maximal} if $|\C H\cap \mathbb{B}[V_1, V_2]|$ does not increase when we move one vertex from one part to another. For example, a partition that maximises $|\C H\cap \mathbb{B}[V_1, V_2]|$ is locally maximal; another (more efficient) way to find one is to start with an arbitrary partition and keep moving vertices one by one between parts as long as each move strictly increases the number of transversal edges.
Note that the following statements hold for any locally maximal partition $V(\C H)=V_1\cup V_2$ by definition:
\begin{enumerate}[label=(\roman*)]
    \item For every $v \in V_1$, we have 
    \begin{align}\label{equ:local-max-link-V1}
         \left| L_{\mathcal{H}}(v) \cap K[V_1, V_2] \right|\ge \left| L_{\mathcal{H}}(v) \cap \binom{V_1}{2} \right|. 
    \end{align}
    \item For every $v \in V_2$, we have 
    \begin{align}\label{equ:local-max-link-V2}
        \left| L_{\mathcal{H}}(v) \cap \binom{V_1}{2} \right|\ge 
        \left| L_{\mathcal{H}}(v) \cap K[V_1, V_2] \right|
        . 
    \end{align}
\end{enumerate}

For a partition $V(\C H)=V_1\cup V_2$, let 
\begin{align}
    B_{\mathcal{H}}(V_1, V_2)
         &\coloneqq \left\{ e \in \C H\colon  |e \cap V_1| \in \{1, 3\}\right\}, \label{equ:def-bad-triple} \\[0.3em]
    M_{\mathcal{H}}(V_1, V_2)
         &\coloneqq  \left\{ e \in \overline{\mathcal{H}} \colon  |e \cap V_1| = 2 \right\}.  \label{equ:def-missing-triple}
\end{align}
be the sets of \textbf{bad} and \textbf{missing} edges respectively. We will omit $(V_1, V_2)$ and the subscript $\mathcal{H}$ if it is clear from the context. If we compare $\C H$ with the recursive construction with the top parts $V_1,V_2$ then,  with respect to the top level, $B$ consists of the additional edges of $\C H$ while $M=\mathbb{B}[V_1,V_2] \setminus \C H$ consists of the top triples not presented in~$\C H$. Note that the edges inside $V_2$ are not classified as bad or missing. 

The key result needed for our proof is the following.

% Given a $3$-graph $\mathcal{H}$, a partition $V_1 \cup V_2 = V(\mathcal{H})$ is called a \textbf{quasi-$2$-colouring} of $\mathcal{H}$ if $\mathcal{H}[V_2] = \emptyset$.
%
\begin{proposition}\label{pr:3}
    For every $\xi > 0$ there exist $\delta_{\ref{pr:3}} = \delta_{\ref{pr:3}}(\xi)$ and $N_{\ref{pr:3}} = N_{\ref{pr:3}}(\xi)$ such that the following holds for all $n \ge N_{\ref{pr:3}}$. 
    Let $\mathcal{H}$ be an $n$-vertex $\CC$-free $3$-graph and $V_1 \cup V_2 = V(\mathcal{H})$ be a locally maximal partition. 
    Suppose that 
    %\begin{enumerate}[label=(\roman*)]
        % \item\label{pr:3-a} $|\mathcal{H}| \ge \be2 \frac{n^3}{6}$, 
        \begin{itemize}
        \item\label{pr:3-b} $\Delta(\mathcal{H}) - \delta(\mathcal{H}) \le \delta_{\ref{pr:3}} n^{2}$, and 
        \item\label{pr:3-c} $|\mathcal{H} \cap \mathbb{B}[V_1, V_2]|\ge \be3 \frac{n^3}{6}$,  
        \end{itemize}
    %\end{enumerate}
     where $\be3\coloneqq 2/5$. Let  $B=B_{\C H}(V_1,V_2)$ and $M=M_{\C H}(V_1,V_2)$ be as defined in~\eqref{equ:def-bad-triple} and~\eqref{equ:def-missing-triple} respectively. Then it holds that
        \begin{align}
        \label{eq:BM}
            |B|-\frac{3999}{4000}\,|M|
            \le \xi n^3.
        \end{align}
\end{proposition}
\begin{proof}[Proof of Proposition~\ref{pr:3}]
    Suppose to the contrary that this proposition fails for some constant $\xi > 0$. 
    Then there exists a sequence $\left(\mathcal{H}_{n}\right)_{n=1}^{\infty}$ of $\CC$-free $3$-graphs of increasing orders such that for every $n \ge 1$: 
    \begin{enumerate}[label=(\roman*)]
        \item\label{proof-pr:3-a} $|\mathcal{H}_n| \ge \be2 \frac{v^3(\mathcal{H}_{n})}{6}$,
        \item\label{proof-pr:3-b} $\frac{\Delta(\mathcal{H}_n) - \delta(\mathcal{H}_n)}{v^{3}(\mathcal{H}_n)} \to 0$ as $n \to \infty$, and
        \item\label{proof-pr:3-c} there exists a locally maximal bipartition $V_1^{n} \cup V_2^{n} = V(\mathcal{H}_{n})$ with $|\mathcal{H}_{n} \cap \mathbb{B}[V_1^{n}, V_2^{n}]|\ge \be3 \frac{v^3(\mathcal{H}_n)}{6}$, 
    \end{enumerate}
    but 
    \begin{align}\label{equ:B-vs-M-flag-algebra-assump}
        |B_n|-\frac{3999}{4000} |M_n|
            > \xi \cdot v^{3}(\mathcal{H}_{n}).
    \end{align}
    where $B_n \coloneqq B_{\mathcal{H}_n}(V_1^{n}, V_{2}^{n})$ and $M_n \coloneqq M_{\mathcal{H}_n}(V_1^{n}, V_{2}^{n})$. 

    By passing to a subsequence, assume that $\left(\mathcal{H}_{n}\right)_{n=1}^{\infty}$ converges to a flag algebra homomorphism $\phi:\C A^0\to \I R$.
    We would like to run flag algebra calculations on the limit $\phi$ in the theory of $\CC$-free \textbf{$2$-vertex-coloured} $3$-graphs (that is, we have two unary relations $V_1, V_2$ with the only restriction that each vertex in a flag satisfies exactly one of them, that is, these unitary relations encode a 2-partition of the vertex set). 

    For $i\in [2]$, let $(1,i)$ denote the $1$-vertex type where the colour of the unique vertex is~$i$.
    Consider the random homomorphism $\boldsymbol{\phi}^{(1,i)}$, which is the limit of taking a uniform random colour-$i$ root. Note that, by~\ref{proof-pr:3-c}, each part $V_i^{n}\subseteq V(\C H_{n})$ is non-empty (in fact, it occupies a positive fraction of vertices), so $\boldsymbol{\phi}^{(1,i)}$ is well-defined.

    The local maximality assumptions~\eqref{equ:local-max-link-V1} and~\eqref{equ:local-max-link-V2} directly translate in the limit to the statement that, 
    \begin{align}
        \boldsymbol{\phi}^{(1,1)}(K_{1,2}^{(1,1)})
        & \ge \boldsymbol{\phi}^{(1,1)}(K_{1,1}^{(1,1)}) \quad\text{and} \label{eq:LM-a}, \\[0.3em]
        \boldsymbol{\phi}^{(1,2)}(K_{1,1}^{(1,2)})
        & \ge \boldsymbol{\phi}^{(1,2)}(K_{1,2}^{(1,2)}). \label{eq:LM-b}
    \end{align}
    with probability 1, where $K_{a,b}^{(1,i)}$ is the $(1,i)$-flag on three vertices that span an edge with the free vertices having colours $a$ and $b$ (and, of course, the root vertex having colour $i$). 
    
    Assumptions~\ref{proof-pr:3-a},~\ref{proof-pr:3-b}, and~\ref{proof-pr:3-c} translate into the following flag algebra inequalities:
    \begin{align}
        & 
        %\sum_{\tau\in \C F_3^3}\phi(E^\tau) \ge \be2, \label{equ:Flag-ineqalities-a} \\[0.3em]
        \phi(E^0) \ge \be2, \label{equ:Flag-ineqalities-a} \\[0.3em]
        & \phi(E_{0}^{\tau} - E_{1}^{\tau}) = 0, \quad  \mbox{for every $2$-vertex type $\tau$}, \label{equ:Flag-ineqalities-b} \\[0.3em]
        & \phi(E_{112}) \ge \be3. \label{equ:Flag-ineqalities-c}
    \end{align}
    where the used terms are as follows. By $E^0$ we mean the sum of all $3$-vertex $0$-flags spanning an edge and being vertex-coloured in all possible ways. (Thus,~\eqref{equ:Flag-ineqalities-a} states that the edge density of the underlying uncoloured $3$-graph is at least $\be2$.) There are 4 different $2$-vertex types $\tau$ with roots $\{0,1\}$, depending on the colours of the roots. For each such type and $i\in \{0,1\}$, $E_{i}^{\tau}$ is the sum of all  $4$-vertex $\tau$-flags such that the two free vertices  make an edge with the root vertex $i$ but not with $1-i$, while the unspecified colours (of the two free vertices) and the unspecified possible 3-edges (those containing both roots) are arbitrary.
     (Thus the combined effect of all restrictions in~\eqref{equ:Flag-ineqalities-b} is  that the underlying uncoloured $3$-graph is almost regular.) Finally, $E_{112}$ denotes the $3$-vertex $0$-flag with one $(1,1,2)$-coloured edge, that is, we look at density of edges with exactly 2 vertices in~$V_1$.
    
    We can now run the usual flag algebra calculations where each of the inequalities in \eqref{equ:Flag-ineqalities-a} and \eqref{equ:Flag-ineqalities-c} (resp.\  \eqref{eq:LM-a}, \eqref{eq:LM-b} and \eqref{equ:Flag-ineqalities-b}) can be multiplied by an unknown non-negative combination of $0$-flags (resp.\ flags of the same type and then averaged out to be a quantum 0-flag), with this added to the SDP program. The final inequality should prove that the left-hand side of~\eqref{eq:BM} is non-positive. Note that the ratio $|M|/\binom{n}{3}$ is the density of the $0$-flag consisting of $(1,1,2)$-coloured non-edge. 
    Likewise, $|B|/\binom{n}{3}$ can be written as the density of the $0$-flag consisting of $(1,1,1)$-coloured and $(1,2,2)$-coloured edges.   
    Now we face the standard flag algebra task of finding the maximum of the quantum $0$-flag expressing $|B|-\frac{3999}{4000}\, |M|$ and checking if it is at most $0$. 

This is by far the most demanding part in terms of the size of the needed semi-definite program. Our proof uses $6$-vertex flags (when the number of linear constraints is $|\C F_6^0|=28080$ for $\CC = \{C_{4}^{3}, C_{5}^{3}\}$, and $|\C F_6^0|=16807$ for $\CC = \{C_{4}^{3}, F_1, F_2\}$). Running it on a conventional PC takes around $18$ and $14$ hours, respectively.
    The results returned by the computer for upper bounding $|B|-\frac{3999}{4000}\, |M|$ are indeed $0$ in both cases (when $\CC = \{C_{4}^{3}, C_{5}^{3}\}$ and $\CC = \{C_{4}^{3}, F_1, F_2\}$). 
    However, by the assumption in~\eqref{equ:B-vs-M-flag-algebra-assump}, $\phi$ satisfies $\phi(M) - \frac{3999}{4000}\, \phi(B) \ge 6\xi >0$, a contradiction. 
\end{proof}

%%%%%%%%%%%%%%%%%%%%%%%%%%%%%%%%%%%%%%%%%%
\section{Tur\'{a}n density}\label{SEC:Proof-turan-density-K4C5}
In this section we prove Theorem~\ref{THM:turan-density-C5K4'} (which implies  Theorem~\ref{THM:turan-density-C5K4}, as observed in Section~\ref{SEC:Prelim}). Here (and in the subsequent sections), we will be rather loose with non-essential constants, making no attempt to optimise them.

The following lemma will be crucial for the proof.
\begin{lemma}\label{LEMMA:recursion-upper-bound}
    % There exists a non-increasing function $N_{\ref{LEMMA:recursion-upper-bound}} \colon (0,1) \to \mathbb{N}$ and a non-decreasing function $\delta_{\ref{LEMMA:recursion-upper-bound}} \colon (0,1) \to (0,1)$ such that for every $\xi > 0$ the following holds for every $n \ge N_{\ref{LEMMA:recursion-upper-bound}}(\xi)$. 
    For every $\xi > 0$, there exist $\delta_{\ref{LEMMA:recursion-upper-bound}} = \delta_{\ref{LEMMA:recursion-upper-bound}}(\xi) > 0$ and $N_{\ref{LEMMA:recursion-upper-bound}} = N_{\ref{LEMMA:recursion-upper-bound}}(\xi)$ such that the following holds for all $n \ge N_{\ref{LEMMA:recursion-upper-bound}}$. 
    Suppose that $\mathcal{H}$ is an $n$-vertex $\CC$-free $3$-graph with at least $\left(\frac{\pi(\CC)}{6}  - \delta_{\ref{LEMMA:recursion-upper-bound}} \right)n^3$ edges.
    Then there exists a partition $V_1 \cup V_2 = V(\mathcal{H})$ with $\frac{n}{2} \le |V_1| \le \frac{4n}{5}$ such that 
    \begin{align}\label{equ:LEMMA:recursion-upper-bound}
        |\mathcal{H}|
        & \le \binom{|V_1|}{2} |V_2| + |\mathcal{H}[V_2]| + \xi n^3 - \max\left\{\frac{|B|}{3999},\,\frac{|M|}{4000}\right\},
    \end{align}
    where $B=B_{\C H}(V_1,V_2)$ and $M=M_{\C H}(V_1,V_2)$ are defined in \eqref{equ:def-bad-triple} and
    \eqref{equ:def-missing-triple} respectively.
\end{lemma}
\begin{proof}[Proof of Lemma~\ref{LEMMA:recursion-upper-bound}] 
    % It is enough to show that, for each sufficiently large integer $m$, say $m\ge m_0$, there is $n_0(m)$ such that the conclusion holds when $\xi=1/m$ and $n\ge n_0(m)$, as then we can take, for example, $N_{\ref{LEMMA:recursion-upper-bound}}(x):=\max\{ n_0(m)\colon m_0\le m\le \lceil{1/x}\rceil\}$ for $x\in (0,1)$.
    % Fix a sufficiently large $m$, let $\xi:=1/m$. 
    Fix $\xi > 0$. We may assume that $\xi$ is sufficiently small. 
    Let $\delta_{\ref{pr:3}} = \delta_{\ref{pr:3}}(\xi)$ and $N_{\ref{pr:3}} = N_{\ref{pr:3}}(\xi)$ be the constants given by Proposition~\ref{pr:3}.  Let $\delta > 0$ be a sufficiently small constant, depending on $\xi$, $\delta_{\ref{pr:3}}$ and $N_{\ref{pr:3}}$.
    %such that $\delta \ll \min\{\delta_{\ref{pr:3}},\,\xi\}$. 
     Let $n$ be sufficiently large such that, in particular, $n \ge 2\,N_{\ref{pr:3}}$. 
    Let $\beta \coloneqq \pi(\CC)$, noting from the $\mathbb{B}_{\mathrm{rec}}$-construction that $\beta \ge \alpha$. 
    Let $\mathcal{H}$ be a $\CC$-free $3$-graph on $n$ vertices with at least $\left(\frac{\beta}{6}  - \delta \right)n^3$ edges. 
    Let 
    \begin{align*}
        V \coloneqq V(\mathcal{H}),\quad 
        Z \coloneqq Z_{\delta}(\mathcal{H}),\quad\text{and}\quad 
        U \coloneqq V(\mathcal{H}) \setminus Z,  
    \end{align*}
    where $Z_{\delta}(\mathcal{H})$ was defined in~\eqref{equ:Z-delta-H-def} for $\mathcal F=\mathcal C'$. 
    
    Let $\mathcal{G}$ denote the induced subgraph of $\mathcal{H}$ on $U$ and let $\hat{n} \coloneqq |U|$. 
    Note from Fact~\ref{FACT:near-extremal-small-deg} that $\hat{n} = n - |Z| \ge (1-\delta^{1/2})n$ and,  
    % \begin{align*}
    %     |\mathcal{H}|
    %     \le |\mathcal{G}| + |Z| \cdot \left(\frac{\beta}{2} - 4\delta^{1/2}\right) n^2 
    %     \le |\mathcal{G}| + \frac{\delta^{1/2} \beta n^2}{2}.
    % \end{align*}
    %
    by the definition of $Z_{\delta}(\mathcal{H})$, we have 
    \begin{align}\label{equ:LEMMA:recursion-upper-bound-G-mindeg}
        \delta(\mathcal{G}) 
        \ge \left(\frac{\beta}{2} - 4\delta^{1/2}\right) n^2 - |Z| n 
        \ge \left(\frac{\beta}{2} - 5\delta^{1/2}\right) n^2
        \ge \left(\frac{\beta}{2} - 5\delta^{1/2}\right) \hat{n}^2. 
    \end{align}
    Combining~\eqref{equ:LEMMA:recursion-upper-bound-G-mindeg} with Proposition~\ref{PROP:C5minus-max-degree}~\ref{PROP:C5minus-max-degree-1}, we see that  
    \begin{align*}
        \Delta(\mathcal{G}) 
        \le \left(\frac{\beta}{2} + 10\left(5 \delta^{1/2}\right)^{1/3}\right) \hat{n}^2.   
    \end{align*}
    It follows that 
    \begin{align}\label{equ:level-one-stability-G-near-regular}
        \Delta(\mathcal{G}) - \delta(\mathcal{G})
        \le \left(\frac{\beta}{2} + 10\left(5 \delta^{1/2}\right)^{1/3}\right) \hat{n}^2 - \left(\frac{\beta}{2} - 5\delta^{1/2}\right) \hat{n}^2 
        \le \delta_{\ref{pr:3}} \hat{n}^2. 
    \end{align}
    Let $V(\mathcal{G}) = U_1\cup U_2$ be a partition such that $|\mathcal{G} \cap \mathbb{B}[U_1, U_2]|$ is maximized.
    Let $V(\mathcal{H}) = V_1 \cup V_2$ be an arbitrary partition such that $U_1 \subseteq V_1$ and $U_2 \subseteq V_2$. 
        Let $x_i \coloneqq |U_i|/\hat{n}$ for $i \in [2]$. 
    Since the partition $U = U_1 \cup U_2$ is maximal (and hence locally maximal), Proposition~\ref{pr:2} (applied to the 3-graph $\mathcal{G}$ on $\hat{n}\ge N_{\ref{pr:3}}$ vertices) implies that 
    \begin{align}\label{equ:LEMMA:recursion-upper-bound-G-cut-ratio}
       |\mathcal{G}[U_1,U_2]| \ge \al2 \frac{\hat{n}^{3}}{6},  
    \end{align}
    where $\al2$ was defined in~\eqref{eq:al2}. 
    By combining~\eqref{equ:LEMMA:recursion-upper-bound-G-cut-ratio} with the trivial upper bound 
    \begin{align*}
        |\mathcal{G}[U_1, U_2]| 
        \le \binom{|U_1|}{2} |U_2| 
        \le \frac{x_1^2(1-x_1)\hat{n}^3}{2}
    \end{align*}
    we obtain through calculations on computer with rational numbers that $\frac{55}{100} \le x_1 \le \frac{78}{100}$, i.e.,\ that $0.55\hat{n} \le |U_1| \le 0.78\hat{n}$.
    Therefore, 
    \begin{align*}
        |V_1|
        & \ge |U_1| 
        \ge 0.55\hat{n} 
        \ge 0.55 \left(1 - \delta^{1/2}\right)n
        \ge \frac{n}{2}
        \quad\text{and} \\
        |V_1|
        & \le |U_1| + |Z|
        \le 0.78 n + \delta^{1/2} n
        \le \frac{4n}{5}. 
    \end{align*}

Recall that $B=B_{\C H}(V_1,V_2)$ and $M=M_{\C H}(V_1,V_2)$. 
    It follows from the definition of $B$ and $M$ that 
    \begin{align*}
        |\mathcal{H}|
        = |\mathcal{H}[V_1, V_2]| + |B| + |\mathcal{H}[V_2]|
        = \binom{|V_1|}{2}|V_2| - |M| + |B| + |\mathcal{H}[V_2]|. 
    \end{align*}
    Therefore, to prove~\eqref{equ:LEMMA:recursion-upper-bound}, it suffices to show that 
    \begin{align}\label{equ:LEMMA:recursion-upper-bound-goal-a}
        |B| - |M| 
        \le \xi n^3 - \max\left\{\frac{|B|}{3999},\,\frac{|M|}{4000}\right\}. 
    \end{align}
    Let $\hat{B} \coloneq B_{\mathcal{G}}[U_1, U_2]$ and $\hat{M} \coloneqq M_{\mathcal{G}}[U_1, U_2]$. 
    Note that $\hat{B} \subseteq B$ and $\hat{M} \subseteq M$, and moreover, 
    \begin{align*}
        |B|
        & \le |\hat{B}| + |Z| \cdot \binom{n}{2}
        \le |\hat{B}| + \frac{\delta^{1/2} n^3}{2} 
        \le |\hat{B}| + \frac{\xi n^3}{10}
        \quad\text{and}\\
        |M|
        & \le |\hat{M}| + |Z| \cdot \binom{n}{2}
        \le |\hat{M}| + \frac{\delta^{1/2} n^3}{2}
        \le |\hat{M}| + \frac{\xi n^3}{10}. 
    \end{align*}
    So~\eqref{equ:LEMMA:recursion-upper-bound-goal-a} is reduced to showing that 
    \begin{align}\label{equ:LEMMA:recursion-upper-bound-goal-b}
        |\hat{B}| - |\hat{M}| 
        \le  \frac{\xi n^3}{2} - \max\left\{\frac{|\hat{B}|}{3999},\,\frac{|\hat{M}|}{4000}\right\}. 
    \end{align}
    Notice that $\mathcal{G}$ satisfies~\eqref{equ:level-one-stability-G-near-regular} and~\eqref{equ:LEMMA:recursion-upper-bound-G-cut-ratio}, and the partition $U_1 \cup U_2 = V(\mathcal{G})$ is locally maximal. So applying Proposition~\ref{pr:3} to $\mathcal{G}$ with $\xi$ there corresponding to $\xi/4$ here, we obtain $|\hat{B}| \le \frac{3999}{4000}\, |\hat{M}| + \frac{\xi \hat{n}^3}{4}$. 
    Consequently,  
    \begin{align*}
        |\hat{B}| - |\hat{M}|
            & = |\hat{B}| - \frac{3999}{4000}\, |\hat{M}| - \frac{|\hat{M}|}{4000}
            \le \frac{\xi \hat{n}^3}{4} - \frac{|\hat{M}|}{4000} \quad\text{and} \\[0.3em]
            |\hat{B}| - |\hat{M}|
            & = \frac{4000}{3999}\left(|\hat{B}| - \frac{3999}{4000}\, |\hat{M}|\right) - \frac{|\hat{B}|}{3999}
            \le \frac{1000\xi \hat{n}^3}{3999} - \frac{|\hat{B}|}{3999},
    \end{align*}
    which implies~\eqref{equ:LEMMA:recursion-upper-bound-goal-b}.
    This completes the proof of Lemma~\ref{LEMMA:recursion-upper-bound}. 
\end{proof}

We are now ready to prove Theorem~\ref{THM:turan-density-C5K4'}.
%(which implies Theorem~\ref{THM:turan-density-C5K4}).
%
\begin{proof}[Proof of Theorem~\ref{THM:turan-density-C5K4'}]  
    Recall that $\alpha$ and $\gamma$ are defined in~\eqref{equ:gamma-alpha-def}.
    %$\alpha = 2\sqrt{3}-3$. 
    Let $\beta \coloneqq \pi(\CC)$. The $\mathbb{B}_{\mathrm{rec}}$-construction from Section~\ref{SEC:Intorduction} shows that $\beta \ge \alpha$ and our task is to show the converse inequality.
    
   Fix an arbitrarily small $\xi > 0$.
    Let $\delta_{\ref{LEMMA:recursion-upper-bound}} = \delta_{\ref{LEMMA:recursion-upper-bound}}(\xi) > 0$ be the constant given by Lemma~\ref{LEMMA:recursion-upper-bound}. By reducing $\delta_{\ref{LEMMA:recursion-upper-bound}}$ if necessary, we may assume that $\delta_{\ref{LEMMA:recursion-upper-bound}} \le \xi$. 
    Let $n$ be sufficiently large and let $\mathcal{H}$ be an $n$-vertex $\CC$-free $3$-graph with $\mathrm{ex}(n,\CC)$ edges, i.e., the maximum possible size. Notice that we can choose $n$ sufficiently large so that  
    \begin{align}\label{equ:turan-number-concentrate}
       \left|\mathrm{ex}(N, \CC) - \beta \, \frac{N^3}{6} \right| 
        \le \delta_{\ref{LEMMA:recursion-upper-bound}} \frac{N^3}6
        \le \xi \frac{N^3}{6},
        \quad\text{for every}~N \ge \frac{n}{5}. 
    \end{align}
    Since $|\mathcal{H}| > \left(\frac{\beta}{6} - \delta_{\ref{LEMMA:recursion-upper-bound}}\right)n^3$, it follows from Lemma~\ref{LEMMA:recursion-upper-bound} that there exists a partition $V_1 \cup V_2 = V(\mathcal{H})$ with $\frac{n}{2} \le |V_1| \le \frac{4n}{5}$ such that 
    \begin{align}\label{equ:recursion-a}
        |\mathcal{H}|
        & \le \binom{|V_1|}{2} |V_2| + |\mathcal{H}[V_2]| + \xi n^3 - \max\left\{\frac{|B|}{3999},\,\frac{|M|}{4000} \right\}.  
        % \notag \\
        % & \le \binom{|V_1|}{2} |V_2| + |\mathcal{H}[V_2]| + \xi n^3. 
    \end{align}
    where $B=B_{\C H}(V_1,V_2)$ and $M=M_{\C H}(V_1,V_2)$ were defined in \eqref{equ:def-bad-triple} and
    \eqref{equ:def-missing-triple} respectively. 
    
    Let $x_i \coloneqq |V_i|/n$ for $i \in [2]$. 
    %Note from~\eqref{equ:turan-number-concentrate} that $|\mathcal{H}[V_2]| \le (\beta + \delta_{\ref{LEMMA:recursion-upper-bound}}) (x_2 n)^3/6$. 
    It follows from~\eqref{equ:turan-number-concentrate} and~\eqref{equ:recursion-a} that 
    \begin{align*}
        (\beta - \xi) \frac{n^3}{6}
        \le |\mathcal{H}|
        \le \binom{|V_1|}{2} |V_2| + |\mathcal{H}[V_2]| + \xi n^3
        \le \frac{x_1^2 x_2 n^3}{2}
            + (\beta + \xi) \frac{(x_2 n)^3}{6}
            + \xi n^3.
    \end{align*}
    Combining this inequality with Fact~\ref{FACT:ineqality} and the fact that $x_2 = 1-x_1 \le 1/2$, we obtain 
    \begin{align*}
        \beta 
        \le \frac{3 x_1^2 x_2}{1-x_2^3} + \frac{(7+x_2^3) \xi}{1-x_2^3}  
        \le \alpha  + \frac{57 \xi}{7}. 
    \end{align*}
    Letting $\xi \to 0$, we obtain that $\pi(\CC) = \beta \le \alpha $, which proves Theorem~\ref{THM:turan-density-C5K4'}. 
\end{proof}

%%%%%%%%%%%%%%%%%%%%%%%%%%%%%%%%%%%%%%%%%%
\section{Stability}\label{SEC:Proof-K4C5-stability}
In this section, we prove Theorem~\ref{THM:C5K4-stability'} (which implies Theorem~\ref{THM:C5K4-stability}).

We begin by establishing the following weaker form of stability. 
\begin{lemma}\label{LEMMA:weak-stability}
    % There exists a non-increasing function $N_{\ref{LEMMA:weak-stability}} \colon (0,1) \to \mathbb{N}$ such that the following holds for every $\xi \in (0, 10^{-8})$ and every $n \ge N_{\ref{LEMMA:weak-stability}}(\xi)$. 
    For every $\xi > 0$, there exist $\delta_{\ref{LEMMA:weak-stability}}  = \delta_{\ref{LEMMA:weak-stability}}(\xi) > 0$ and  $N_{\ref{LEMMA:weak-stability}} = N_{\ref{LEMMA:weak-stability}}(\xi)$ such that the following holds for all $n \ge N_{\ref{LEMMA:weak-stability}}$. 
    Suppose that $\mathcal{H}$ is an $n$-vertex $\CC$-free $3$-graph with at least $\left(\frac{\alpha}{6} - \delta_{\ref{LEMMA:weak-stability}} \right) n^3$ edges.
    Then there exists a partition $V_1 \cup V_2 = V(\mathcal{H})$ such that 
    \begin{enumerate}[label=(\roman*)]
        \item\label{LEMMA:weak-stability-1}  $\left||V_1| -\gamma n \right| \le \sqrt{12 \xi} n$, 
        \item\label{LEMMA:weak-stability-2} $\max\left\{|B|,\,|M|\right\} \le 12000 \xi n^3$, where $B=B_{\mathcal{H}}(V_1,V_2)$ and $M=M_{\mathcal{H}}(V_1,V_2)$ are defined in~\eqref{equ:def-bad-triple} and~\eqref{equ:def-missing-triple},  and
        \item\label{LEMMA:weak-stability-3} $|\mathcal{H}[V_2]| \ge \left(\frac{\alpha}{6} - 1500\xi \right) |V_2|^3$. 
    \end{enumerate}
\end{lemma}
\begin{proof}[Proof of Lemma~\ref{LEMMA:weak-stability}]
%
% As in the proof of Lemma~\ref{LEMMA:recursion-upper-bound}, there is a non-increasing function $N':(0,1)\to \mathbb{N}$ such that~\eqref{equ:turan-number-concentrate} holds for every $\xi\in (0,1)$ and $n\ge N'(\xi)$.   
% Let 
% \begin{align*}
%     N_{\ref{LEMMA:weak-stability}}(\xi)
%     \coloneqq \max(N'(\xi), N_{\ref{LEMMA:recursion-upper-bound}}(\xi))
%     \quad\mbox{for $\xi\in (0,10^{-8})$}.
% \end{align*}
% %
% Take any $\xi\in (0,10^{-8})$ and $n\ge N_{\ref{LEMMA:weak-stability}}(\xi)$. 
% Let $\C H$ be as in the lemma. 
Fix $\xi > 0$. We may assume that $\xi$ is sufficiently small. 
Let $\delta_{\ref{LEMMA:recursion-upper-bound}} = \delta_{\ref{LEMMA:recursion-upper-bound}}(\xi) > 0$ and $N_{\ref{LEMMA:recursion-upper-bound}} = N_{\ref{LEMMA:recursion-upper-bound}}(\xi)$ be constants given by Lemma~\ref{LEMMA:recursion-upper-bound}. 
By reducing $\delta_{\ref{LEMMA:recursion-upper-bound}}$ if necessary, we may assume that $\delta_{\ref{LEMMA:recursion-upper-bound}} \le \xi$. 
Let $n$ be sufficiently large and let $\mathcal{H}$ be an $n$-vertex $\CC$-free $3$-graph with at least $\left(\frac{\alpha}{6} - \delta_{\ref{LEMMA:recursion-upper-bound}} \right) n^3$ edges.  
Similarly to the proof of Theorem~\ref{THM:turan-density-C5K4'}, we choose $n$ sufficiently large so that~\eqref{equ:turan-number-concentrate} holds. 

By Lemma~\ref{LEMMA:recursion-upper-bound}, there exists a partition $V(\mathcal{H}) = V_1 \cup V_2$ with $n/2 \le |V_1| \le 4n/5$ such that~\eqref{equ:LEMMA:recursion-upper-bound} holds.
Let $x_i \coloneqq |V_i|/n$ for $i \in [2]$. 
Combining~\eqref{equ:LEMMA:recursion-upper-bound} and~\eqref{equ:turan-number-concentrate} with the assumption $|\mathcal{H}| \ge \left(\frac{\alpha}{6} - \delta_{\ref{LEMMA:recursion-upper-bound}} \right) n^3 \ge \left(\frac{\alpha}{6} - \xi \right) n^3$, we obtain 
\begin{align}
    \left(\frac{\alpha}{6} - \xi\right) n^3
    \le |\mathcal{H}|
    & \le \frac{x_1^2 x_2 n^3}{2} + |\mathcal{H}[V_2]| + \xi n^3 - \max\left\{\frac{|B|}{3999},\,\frac{|M|}{4000}\right\} \notag \\
    & \le \frac{x_1^2 x_2 n^3}{2} + \left(\alpha  + \xi \right) \frac{(x_2 n)^3}{6} + \xi n^3  - \max\left\{\frac{|B|}{3999},\,\frac{|M|}{4000}\right\}  \notag \\
    & \le \left(\frac{x_1^2 x_2}{2} + \frac{\alpha}{6} x_2^3\right) n^3 + 2\xi n^3 - \max\left\{\frac{|B|}{3999},\,\frac{|M|}{4000}\right\} \notag \\
    & \le \frac{\alpha}{6}\, n^3 - \frac{(x_1-\gamma)^2}{4}\, n^3 + 2\xi n^3 - \max\left\{\frac{|B|}{3999},\,\frac{|M|}{4000}\right\}, \label{equ:LEMMA:recursion-upper-bound-2}
\end{align}
where the last inequality follows from Fact~\ref{FACT:ineqality}~\ref{FACT:ineqality-2}.

It follows from~\eqref{equ:LEMMA:recursion-upper-bound-2} that $\frac{(x_1-\gamma)^2}{4} \le 3\xi$ and $\max\left\{\frac{|B|}{3999},\,\frac{|M|}{4000}\right\} \le 3\xi n^3$, implying respectively Conclusions~\ref{LEMMA:weak-stability-1} and~\ref{LEMMA:weak-stability-2} of Lemma~\ref{LEMMA:weak-stability}. 

    Next, we prove Lemma~\ref{LEMMA:weak-stability}~\ref{LEMMA:weak-stability-3}.
    Suppose to the contrary that $|\mathcal{H}[V_2]| \le \left(\frac{\alpha}{6} - 1500 \xi\right) x_2^3 n^3$. 
    Then, similarly to the proof above, it follows from~\eqref{equ:LEMMA:recursion-upper-bound} and Fact~\ref{FACT:ineqality}~\ref{FACT:ineqality-2} that  
        \begin{align*}
            \left(\frac{\alpha}{6} - \xi\right) n^3 
            & \le \frac{x_1^2 x_2 n^3}{2} + \left(\frac{\alpha}{6} - 1500 \xi\right) x_2^3 n^3 + \xi n^3  \\
            & = \left(\frac{x_1^2 x_2}{2} + \frac{\alpha}{6} x_2^3\right) n^3 - \frac{1500 \xi x_2^3 n^3}{6} + \xi n^3  \\
            & \le \frac{\alpha}{6} n^3 - \frac{1500 \xi (1/5)^3 n^3}{6} + \xi n^3
            < \left(\frac{\alpha}{6} - \xi\right) n^3, 
        \end{align*}
        a contradiction. Here, in the second to the last inequality, we used the fact that $x_2 = 1-x_1 \ge 1/5$. 
        % This completes the proof of Lemma~\ref{LEMMA:weak-stability}. 
\end{proof}%PROP

Now we are ready to prove Theorem~\ref{THM:C5K4-stability'}. 
\begin{proof}[Proof of Theorem~\ref{THM:C5K4-stability'}]
    Fix $\varepsilon > 0$. We may assume that $\varepsilon$ is sufficiently small. 
    Let $k \in \mathbb{N}$ be an integer such that $(1/2)^{k-1} \le \varepsilon$. 
    Let $0 < \delta \coloneqq \varepsilon_0 \ll \varepsilon_1 \ll \ldots \ll \varepsilon_k \ll \varepsilon$ be sufficiently small constants such that, in particular, $\varepsilon_{i-1} \le \min\left\{\delta_{\ref{LEMMA:weak-stability}}(\varepsilon_{i}/12000),\,\varepsilon/2^{k-i+2} \right\}$ for $i \in [k]$, where $\delta_{\ref{LEMMA:weak-stability}} \colon (0,1) \to (0,1)$ is the function given by Lemma~\ref{LEMMA:weak-stability}.
    
    Let $n$ be sufficiently large; in particular, we can assume that 
    \begin{align*}
        \frac{n}{5^k}
        \ge \max\left\{N_{\ref{LEMMA:weak-stability}}\left(\frac{\varepsilon_i}{12000}\right) \colon i \in [k]\right\}, 
    \end{align*}
    where $N_{\ref{LEMMA:weak-stability}} \colon (0,1) \to \mathbb{N}$ is the function given by Lemma~\ref{LEMMA:weak-stability}.

    Let us prove by a backward induction on $i$ that for every $i \in [k]$ and $m \in \left[n/5^{i-1}, n/2^{i-1}\right]$, every $m$-vertex $\CC$-free $3$-graph $\C H$ with least $\left(\frac{\alpha}{6} - \varepsilon_{i-1} \right)m^3$ edges can be transformed into a $\mathbb{B}_{\mathrm{rec}}$-subconstruction by removing at most $(\varepsilon_i + \cdots + \varepsilon_k + \varepsilon/6) n^3$ edges.
    
    The base case $i = k$ is trivially true since, by assumption, we have $m \le n/2^{k-1} \le \varepsilon n$, and hence, $\binom{m}{3} \le \varepsilon n^3/6$.
    So we may assume that $i \le k-1$. 

    Fix $i \in [k-1]$ and fix an arbitrary integer $m \in \left[n/5^{i-1}, n/2^{i-1}\right]$.
    Let $\mathcal{H}$ be an  $m$-vertex $\CC$-free $3$-graph with at least $\left(\frac{\alpha}{6} - \varepsilon_{i-1} \right)m^3$ edges. 
    Since 
    \begin{align*}
        m 
        \ge n/5^{i-1} 
        \ge n/5^k 
        \ge N_{\ref{LEMMA:weak-stability}}(\varepsilon_{i}/12000)
        \quad\text{and}\quad 
        \varepsilon_{i-1} 
        \le \delta_{\ref{LEMMA:weak-stability}}(\varepsilon_{i}/12000), 
    \end{align*}
    applying Lemma~\ref{LEMMA:weak-stability} to $\mathcal{H}$ with $\xi$ there corresponding to $\varepsilon_i/12000$ here, we obtain a partition $V(\mathcal{H}) = V_1 \cup V_2$ such that 
    \begin{enumerate}[label=(\roman*)]
        \item\label{item:Vi-size} $|V_1| \in \left[\left(\gamma - \sqrt{\frac{12 \varepsilon_{i}}{12000}}\right) m, \left(\gamma + \sqrt{\frac{12 \varepsilon_{i}}{12000}}\right) m\right] \subseteq \left[\frac{m}{2}, \frac{4m}{5}\right]$, 
        \item\label{item:B-size} $|B_{\mathcal{H}}(V_1, V_2)| \le 12000 \cdot \frac{\varepsilon_i}{12000} m^3 = \varepsilon_i m^3 \le \varepsilon_i n^3$, and 
        \item\label{item:H-Vi-size} $|\mathcal{H}[V_2]| \ge \left(\frac{\alpha}{6} - 1500 \cdot \frac{\varepsilon_i}{12000} \right) |V_2|^3 \ge \left(\frac{\alpha}{6} - \varepsilon_i \right) |V_2|^3$. 
    \end{enumerate}
    It follows from~\ref{item:Vi-size} that $|V_2| \in \left[m/5, m/2\right] \subseteq \left[n/5^{i}, n/2^{i}\right]$. 
    So, by~\ref{item:H-Vi-size} and the inductive hypothesis, $\mathcal{H}[V_2]$ is a $\mathbb{B}_{\mathrm{rec}}$-subconstruction after removing at most $\left(\varepsilon_{i+1} + \cdots + \varepsilon_{k} + \varepsilon/6\right)n^3$ edges. 
    Combining it with~\ref{item:B-size}, we see that $\mathcal{H}$ is a $\mathbb{B}_{\mathrm{rec}}$-subconstruction after removing at most 
    \begin{align*}
        |B_{\mathcal{H}}(V_1, V_2)| +\left(\varepsilon_{i+1} + \cdots + \varepsilon_{k} + \varepsilon/6\right)n^3 
        \le \left(\varepsilon_{i} + \cdots + \varepsilon_{k} + \varepsilon/6\right)n^3
    \end{align*}
    edges. 
    This completes the proof for the inductive step. 
    
    Taking $i = 1$ and $m = n$, we obtain that every $n$-vertex $\CC$-free $3$-graph with at least $\left(\frac{\alpha}{6} - \varepsilon_0 \right) n^3$ edges is a $\mathbb{B}_{\mathrm{rec}}$-subconstruction after removing at most 
    \begin{align*}
        \left(\varepsilon_{1} + \cdots + \varepsilon_{k} + \frac{\varepsilon}{6}\right)n^3 
        \le \left(\sum_{i=1}^{k} \frac{\varepsilon}{2^{k-i+2}} + \frac{\varepsilon}{6}\right) n^3
        \le \varepsilon n^3
    \end{align*}
    edges.
    This completes the proof of Theorem~\ref{THM:C5K4-stability'}.
\end{proof}
%%%%%%%%%%%%%%%%%%%%%%%%%%%%%%%%%%%%%%%%%%%%

%%%%%%%%%%%%%%%%%%%%%%%%%%%%%%%%%%%%%%%%%%
\section{Finer structure}\label{SEC:proof-C5-exact}
In this section, we prove Theorem~\ref{THM:K4C5-first-level-exact}. 
First, we present a preliminary result for graphs. 

% Recall from~\eqref{equ:gamma-alpha-def} that $\alpha \coloneqq 2\sqrt{3}-3$ and $\gamma \coloneqq \frac{3-\sqrt{3}}{2}$.
We represent a $4$-vertex path on the set $\{a, b, c, d\}$ with the edge set $\{ab, bc, cd\}$ using the ordered $4$-tuple $(a,b,c,d)$. 
Let $G$ be a graph, and let $V_1 \cup V_2 = V(G)$ be a partition of $V(G)$.  
We say a $4$-vertex path $(a,b,c,d)$ in $G$ is  
\begin{enumerate}[label=(\roman*)]
    \item of \textbf{type-1} (with respect to $(V_1, V_2)$) if $\{a,b,c\} \subseteq V_1$ and $d \in V_2$, and  
    \item of \textbf{type-2} (with respect to $(V_1, V_2)$) if $\{a,d\} \subseteq V_1$ and $\{b,c\} \subseteq V_2$.  
\end{enumerate}

\begin{proposition}\label{PROP:P4-free}
    For every $\xi > 0$, there exist $\delta > 0$ and $n_0$ such that the following holds for all $n \ge n_0$. 
    Suppose that $G$ is an $n$-vertex graph and $V_1 \cup V_2 = V(G)$ is a partition such that 
    \begin{enumerate}[label=(\roman*)]
        \item\label{PROP:P4-free-1} the number of both type-1 and type-2 four-vertex paths in $G$ is at most $\delta n^4$, 
        \item\label{PROP:P4-free-2} $\left| |V_1| - \gamma n \right| \le \delta n$, and 
        \item\label{PROP:P4-free-3} $|G[V_2]| \le \left(\frac{\alpha}{2} + \delta \right) |V_2|^2$.  
    \end{enumerate}
    Then $|G| \le \left(\frac{\alpha}{2} + \xi \right) n^2$. 
    Moreover, if $|G| \ge \left(\frac{\alpha}{2} - \delta \right) n^2$, then 
    \begin{align*}
        \mbox{either}\quad 
        |G[V_1, V_2]| \le \xi n^2
        \quad\mbox{or}\quad 
        |G[V_1]| + |G[V_2]| \le \xi n^2. 
    \end{align*}
\end{proposition}

We will use the following inequality in the proof of Proposition~\ref{PROP:P4-free}. 
\begin{proposition}\label{PROP:inequality-fxy}
    Let  
    \begin{align}\label{equ:fxy-def}
        f(x,y)
        \coloneqq \frac{x^2}{2} + \left(\gamma - x\right)(1- \gamma - y) + \min\left\{\frac{y^2}{2}+y(1-\gamma-y),\,\frac{\alpha}{2} (1-\gamma)^2\right\}. 
    \end{align}
    Then
    \begin{align}\label{equ:fxy-max}
        \max\left\{f(x,y) + \frac{\left(\gamma-x\right) \left(x+y\right)}{10} \colon (x,y) \in [0, \gamma] \times [0, 1-\gamma]\right\}
        \le \frac{\alpha}{2}. 
    \end{align} 
\end{proposition}
\begin{proof}[Proof of Proposition~\ref{PROP:inequality-fxy}]
    Let $g(x,y) \coloneqq f(x,y) + \frac{\left(\gamma-x\right) \left(x+y\right)}{10}$. 
    Notice that for every fixed $y \in [0,1-\gamma]$, $g(x,y)$ is a quadratic polynomial in $x$ with a leading coefficient of $\frac{1}{2} - \frac{1}{10} > 0$. This implies that, for fixed $y$, $g(x,y)$ attains its maximum at the endpoints of the interval $[0, \gamma]$, meaning that $g(x,y) \le \max\left\{g(0, y),\,g(\gamma, y)\right\}$ for every $y \in [0, 1-\gamma]$. 
    Therefore, to prove~\eqref{equ:fxy-max}, it suffices to show that 
    \begin{align*}
        \max\left\{g(0, y),\,g(\gamma, y)\right\} 
        \le \frac{\alpha}{2}. 
    \end{align*}
    Straightforward calculations show that $g(0, y)$ is decreasing in $y$ and satisfies $g(0,0) = \alpha/2$, while $g(\gamma, y)$ is increasing in $y$ and satisfies $g(\gamma,1-\gamma) = \alpha/2$. Therefore, we have $g(0, y) \le \alpha/2$ and $g(\gamma, 1-\gamma) \le \alpha/2$. 
    This completes the proof of Proposition~\ref{PROP:inequality-fxy}. 
\end{proof}

We are now ready to prove Proposition~\ref{PROP:P4-free}. 
\begin{proof}[Proof of Proposition~\ref{PROP:P4-free}]
    Fix $\xi > 0$. Choose sufficiently small constants $\delta, \varepsilon > 0$ such that $\delta \ll \varepsilon \ll \xi$. 
    Let a $3$-graph $G$ of sufficiently large order $n$ and a partition $V_1 \cup V_2 = V(G)$ be as specified in  Proposition~\ref{PROP:P4-free}. 
    % By moving at most $\delta n$ vertices from $V_1$ to $V_2$, we may assume that $|V_1| \le \gamma n$. 

If $|G[V_1]| \le \varepsilon n^2$, then define $X \coloneqq \emptyset$. Otherwise, applying Fact~\ref{FACT:min-deg-cleaning} to $G[V_1']$, where $V_1'$ is some subset of $V_1$ of size $\min\{\gamma n,|V_1|\}$, we obtain a set $X \subseteq V_1'$ such that the minimum degree of the induced subgraph $G[X]$ is at least $\varepsilon n$ and 
    $|G[X]|\ge |G[V_1']|-\e n|V_1|$. The last inequality implies that
    \begin{align}\label{equ:G-V1-upper-bound}
        |G[V_1]|
        & \le |G[X]| + \varepsilon n|V_1|+n|V_1\setminus V_1'| 
        \le \frac{|X|^2}{2} + 2\varepsilon n^2. 
    \end{align}
    
    Similarly, if $|G[V_1, V_2]| \le \varepsilon n^2$, then define $Z_1 = Z_2 \coloneqq \emptyset$. Otherwise, applying Fact~\ref{FACT:min-deg-cleaning} to the bipartite graph $G[V_1, V_2]$, we obtain sets $(Z_1, Z_2) \subseteq V_1 \times V_2$ such that the minimum degree of the induced bipartite subgraph  $G[Z_1, Z_2]$ is at least $\varepsilon n$, and moreover, 
    \begin{align}\label{equ:G-V1-V2-upper-bound}
        |G[V_1,V_2]|
        & \le |G[Z_1, Z_2]| + \varepsilon n^2
        \le |Z_1||Z_2| + \varepsilon n^2. 
    \end{align}
    \begin{claim}\label{CLAIM:XYZ-disjoint}
        We have $|X \cap Z_1| \le \varepsilon n$ and $|G[Z_2]| \le \varepsilon n^2$. 
    \end{claim}
    \begin{proof}[Proof of Claim~\ref{CLAIM:XYZ-disjoint}]
        First, suppose to the contrary that $|X \cap Z_1| \ge \varepsilon n$. 
        By the definitions of $X$ and $Z_1$, it is easy to see from a simple greedy argument that the number of injective embeddings $\psi$ of a $4$-vertex path $(a,b,c,d)$ into $G$ with $\psi(c) \in X\cap Z$, $\psi(d) \in N_{G[Z_1, Z_2]}(\psi(c))$, $\psi(b) \in N_{G[X]}(\psi(c))$, and $\psi(a) \in N_{G[X]}(\psi(b)) \setminus \{\psi(c)\}$ is at least  
        \begin{align*}
            |X \cap Z| \cdot \varepsilon n \cdot \varepsilon n \cdot (\varepsilon n - 1)
            \ge \varepsilon^4 n^4/2
            > \delta n^4,
        \end{align*}
        contradicting the assumption that the number of  4-vertex paths of type-1 is at most $\delta n^4$. 

        Now suppose to the contrary that $|G[Z_2]| \ge \varepsilon n^2$. 
        By the definition of $Z_2$, it is easy to see from a simple greedy argument that the number of injective embeddings $\psi$ of a $4$-vertex path $(a,b,c,d)$ into $G$ with $\{\psi(b),\psi(c)\} \in G[Z_2]$, $\psi(a) \in N_{G[Z_1, Z_2]}(\psi(b))$, and $\psi(d) \in N_{G[Z_1, Z_2]}(\psi(c)) \setminus \{\psi(a)\}$ is at least  
        \begin{align*}
            |G[Z_2]| \cdot \varepsilon n \cdot  (\varepsilon n - 1)
            \ge \varepsilon^3 n^4/2
            > \delta n^4,
        \end{align*}
        contradicting the assumption that the number of type-2 four-vertex paths is at most $\delta n^4$.
        
        This completes the proof of Claim~\ref{CLAIM:XYZ-disjoint}. 
    \end{proof}%CLAIM

    Let $Y \coloneqq V_{2} \setminus Z_2$. 
    Let 
    \begin{align*}
        x \coloneqq |X|/n, \quad 
        y \coloneqq \min\left\{ |Y|/n,\,1-\gamma \right\}, \quad
        z_1 \coloneqq |Z_1|/n, \quad\text{and}\quad 
        z_2 \coloneqq |Z_2|/n. 
    \end{align*}
    
    It follows from Claim~\ref{CLAIM:XYZ-disjoint} and assumption~\ref{PROP:P4-free-2} that 
    \begin{align}
        z_1 
        & \le \frac{|V_1| - |X| + \varepsilon n}{n}
        \le \gamma - x + 2 \varepsilon, \quad\text{and} \label{equ:z-upper-bound} \\[0.3em]
        |G[V_2]| 
        & = |G[Y]|  + |G[Y, Z_2]| + |G[Z_2]|
        \le \frac{y^2 n^2}{2}  + y(1-\gamma -y)n^2 + 2 \varepsilon n^2. \label{equ:G-V2-upper-bound}
    \end{align}
    Combining~\eqref{equ:G-V1-upper-bound},~\eqref{equ:G-V1-V2-upper-bound},~\eqref{equ:z-upper-bound},~\eqref{equ:G-V2-upper-bound}, and Assumption~\ref{PROP:P4-free-3}, we obtain 
    \begin{align}
        \frac{|G|}{n^2}
        & \le \frac{x^2}{2} +  (\gamma - x)(1-\gamma-y) + \min\left\{\frac{y^2}{2}+y(1-\gamma -y),\,\frac{\alpha}{2}(1-\gamma)^2\right\} + 7 \varepsilon \notag \\[0.3em]
        & = f(x,y) + 7 \varepsilon
        \le \frac{\alpha}{2}  - \frac{\left(\gamma - x \right) \left(x + y\right)}{10}  + 7 \varepsilon. \label{equ:G-upper-bound-fxy}
    \end{align}
    where $f(x,y)$ was defined in~\eqref{equ:fxy-def}, and the last inequality follows from Proposition~\ref{PROP:inequality-fxy}.

    Since $\gamma - x \ge 0$ and $x+y \ge 0$, it follows from~\eqref{equ:G-upper-bound-fxy} that $|G| \le \frac{\alpha}{2} n^2 + 7 \varepsilon n^2 \le \frac{\alpha}{2} n^2 + \xi n^2$, proving the first part of Proposition~\ref{PROP:P4-free}.
    
    For the second part, suppose additionally that $|G| \ge \frac{\alpha}{2} n^2 - \delta n^2$. Then, by~\eqref{equ:G-upper-bound-fxy}, we have 
    \begin{align*}
        \frac{\left(\gamma-x \right) \left(x+y\right)}{10} 
        \le 7 \varepsilon + \delta 
        \le 10 \varepsilon,  
    \end{align*}
    which implies that 
    \begin{align*}
        \min\left\{\gamma - x,\,x+y\right\} \le 10 \sqrt{\varepsilon}. 
    \end{align*}
    Suppose that $|x - \gamma| \le 10 \sqrt{\varepsilon}$. 
    Then it follows from~\eqref{equ:z-upper-bound} that 
    \begin{align*}
        |Z_1| 
        = zn 
        \le \left(10\sqrt{\varepsilon} + 2\varepsilon\right) n 
        \le 12\sqrt{\varepsilon} n, 
    \end{align*}
    which, combined with~\eqref{equ:G-V1-V2-upper-bound}, implies that 
    \begin{align*}
        |G[V_1, V_2]|
        \le |G[Z_1, Z_2]| + \varepsilon n^2 
        \le 12\sqrt{\varepsilon} n^2 + \varepsilon n^2
        \le \xi n^2.
    \end{align*}
    Suppose that $x+y \le 10 \sqrt{\varepsilon}$. Then it follows from~\eqref{equ:G-V1-upper-bound} and~\eqref{equ:G-V2-upper-bound} that 
    \begin{align*}
        |G[V_1]| + |G[V_2]|
        & \le \frac{x^2 n^2}{2} + 2\varepsilon n^2  + \frac{y^2 n^2}{2}  + y(1-\gamma -y)n^2 + 4 \varepsilon n^2  \\
        & \le 10 \sqrt{\varepsilon} n^2 + 2\varepsilon n^2 
        \le \xi n^2. 
    \end{align*}
    This completes the proof of Proposition~\ref{PROP:P4-free}. 
\end{proof}%PROP

We are now ready to prove Theorem~\ref{THM:K4C5-first-level-exact}. Recall that it states that every maximum $\{C_{4}^{3}, C_{5}^{3}\}$-free $3$-graph of sufficiently large order admits a vertex partition $V_1\cup V_2$ so that all edges, apart from those inside $V_2$, are exactly as stipulated by the $\mathbb{B}_{\mathrm{rec}}$-construction.
\begin{proof}[Proof of Theorem~\ref{THM:K4C5-first-level-exact}]
    Fix $\varepsilon >0$. We may assume that $\varepsilon$ is small. 
    Choose constants $\delta, \delta_1, \ldots, \delta_{5} \in (0, 10^{-8})$ to be sufficiently small such that $\delta \ll \delta_1 \ll \cdots \ll \delta_{5} \ll \varepsilon$, and let $n \gg 1/\delta$ be sufficiently large. 
    Let $\mathcal{H}$ be an $n$-vertex $\{C_{4}^{3}, C_{5}^{3}\}$-free $3$-graph with $\mathrm{ex}(n,\{C_{4}^{3}, C_{5}^{3}\})$ edges, i.e., the maximum possible size. Since $n$ is sufficiently large and $\pi(\{C_{4}^{3}, C_{5}^{3}\}) = \alpha $, we have $|\mathcal{H}| \ge \left(\frac{\alpha}{6} - \delta\right)n^3$. 
    Since $n$ is sufficiently large, we have by Proposition~\ref{PROP:C5minus-max-degree} that 
    \begin{align}\label{equ:exact-min-max-deg}
        \left(\frac{\alpha}{2} - \delta \right) n^2 
        \le \delta(\mathcal{H})
        \le \Delta(\mathcal{H})
        \le \left(\frac{\alpha}{2} + \delta \right) n^2. 
    \end{align}
    
    Let $V \coloneqq V(\mathcal{H})$. 
    Let $V_1 \cup V_2 = V$ be the locally maximal partition returned by Lemma~\ref{LEMMA:weak-stability} and let $x_i \coloneqq |V_i|/n$ for $i \in [2]$. 
    Let 
    \begin{align*}
        \mathcal{G} \coloneqq \mathcal{H} \cap \mathbb{B}[V_1, V_2], 
        \quad 
        B \coloneqq B_{\C H}(V_1,V_2), \quad\text{and}\quad 
        M \coloneqq M_{\C H}(V_1,V_2), 
    \end{align*}
    where $B_{\C H}(V_1,V_2)$ and $M_{\C H}(V_1,V_2)$ are respectively the sets of bad edges and missing edges, as defined in~\eqref{equ:def-bad-triple} and~\eqref{equ:def-missing-triple}. 
    % It follows from Theorem~\ref{THM:C5K4-stability} that 

By Lemma~\ref{LEMMA:weak-stability} and $\delta\ll \delta_1$,
%~\ref{LEMMA:weak-stability-1},~\ref{LEMMA:weak-stability-2}, and~\ref{LEMMA:weak-stability-3}    
we can assume that the following inequalities hold:
    \begin{align}
        & \max\left\{\, \left| x_1 - \gamma \right|,\,\left| x_2 - (1-\gamma)  \right|\,\right\} \le \delta_{1}, \label{equ:vtx-stab-a} \\[0.3em]
        & \max\left\{\,|B|,\,|M|\,\right\} 
          \le \delta_{1} n^{3}, \quad\text{and} \label{equ:vtx-stab-b}  \\[0.3em]
        % & |\mathcal{H} \cap \mathbb{B}[V_1, V_2]| 
        %  = \binom{|V_1|}{2}|V_2| - |M|
        % \ge \frac{|V_1|^2 |V_2|}{2} - \delta_{1} n^3, \quad\text{and} \label{equ:vtx-stab-c} \\[0.3em]
        & |\mathcal{H}[V_2]|
        \ge \left(\frac{\alpha}{6} - \delta_{1}\right)|V_2|^3. \label{equ:vtx-stab-d}
    \end{align}

    For convenience, for every vertex $v \in V$, let $L(v) \coloneqq L_{\mathcal{H}}(v)$. 
    For $i \in [2]$, let $L_{i}(v)$ denote the induced subgraph of $L(v)$ on $V_i$.
    Let $L_{12}(v)$ denote the induced bipartite subgraph of $L(v)$ with parts $V_1$ and $V_2$.

    % \begin{claim}\label{CLAIM:max-deg-G-V2}
    %     For every vertex $v \in V$, we have $|L_2(v)| \le \left(\frac{\alpha}{2} + \delta_{2}\right)|V_2|^2$. 
    % \end{claim}
    % \begin{proof}[Proof of Claim~\ref{CLAIM:max-deg-G-V2}]
    %     It follows from~\eqref{equ:vtx-stab-d} and Proposition~\ref{PROP:C5minus-max-degree-1} that 
    %     \begin{align}
    %         ?
    %     \end{align}
    % \end{proof}

    \begin{claim}\label{CLAIM:B-M-max-deg}
        For every vertex $v\in V$, we have $d_{M}(v) \le d_{B}(v) + \delta_{2} n^2$. 
    \end{claim}
    \begin{proof}[Proof of Claim~\ref{CLAIM:B-M-max-deg}]
        Suppose that $v\in V_1$. Then we have 
        \begin{align*}
            d_{\mathcal{H}}(v) 
            = d_{\mathbb{B}[V_1, V_2]}(v) + d_{B}(v) - d_{M}(v)
            = \left(|V_1| - 1\right)|V_2| + d_{B}(v) - d_{M}(v). 
        \end{align*}
        Combining it with~\eqref{equ:exact-min-max-deg} and~\eqref{equ:vtx-stab-a}, we obtain  
        \begin{align*}
            |d_{B}(v) - d_{M}(v)|
             % = \left| d_{\mathcal{H}}(v) -  d_{\mathbb{B}[V_1, V_2]}(v) \right| 
             = \left| d_{\mathcal{H}}(v) -  (x_1 n-1) x_2n \right|
            \le \delta_{2} n^2.  
        \end{align*}

        Now consider the case where $v \in V_2$.
        It follows from~\eqref{equ:vtx-stab-d} and Proposition~\ref{PROP:C5minus-max-degree}~\ref{PROP:C5minus-max-degree-1} that  
        \begin{align*}
            d_{\mathcal{H}[V_2]}(v)
            \le \Delta(\mathcal{H}[V_2])
            \le \left(\frac{\alpha}{2} + \frac{\delta_{2}}{2}\right)|V_2|^2. 
        \end{align*}
        Combining it with the identity
        \begin{align*}
            d_{\mathcal{H}}(v)
            & = d_{\mathbb{B}[V_1, V_2]}(v) + d_{B}(v) - d_{M}(v) + d_{\mathcal{H}[V_2]}(v) \\
            & = \binom{|V_1|}{2} + d_{B}(v) - d_{M}(v) + d_{\mathcal{H}[V_2]}(v), 
        \end{align*}
        we obtain 
        \begin{align*}
            d_{M}(v) - d_{B}(v)
            =  \binom{|V_1|}{2} + d_{\mathcal{H}[V_2]}(v) - d_{\mathcal{H}}(v)
            \le \binom{|V_1|}{2} + \left(\frac{\alpha}{2} + \frac{\delta_{2}}{2} \right)|V_2|^2  - d_{\mathcal{H}}(v).
        \end{align*}
        Combining this inequality with~\eqref{equ:exact-min-max-deg} and~\eqref{equ:vtx-stab-a}, we obtain $d_{M}(u) - d_{B}(v) \le \delta_2 n^2$. 
        This completes the proof of Claim~\ref{CLAIM:B-M-max-deg}. 
    \end{proof}%CLAIM

    We partition $B$ into two subfamilies as follows: 
    \begin{align*}
        B_{111}
        \coloneqq \left\{e\in B \colon e\subseteq V_1\right\} 
        \quad\text{and}\quad 
        B_{122}
        \coloneqq \left\{e\in B \colon |e\cap V_1| = 1\right\}, 
    \end{align*}
    A pair $(u,v) \in V_1 \times V_2$  \textbf{heavy} if $d_{B}(uv) \ge \delta_{4} n$. 
    We further partition $B_{122}$ into two subfamilies: 
    \begin{align*}
        B_{122}^{\heavy}
        \coloneqq \left\{e\in B_{122} \colon \text{$e$ contains at least one heavy pair}\right\} 
        \quad\text{and}\quad 
        B_{122}^{s} 
        \coloneqq B_{122} \setminus B_{122}^{\heavy}.
    \end{align*}
    The following fact follows directly from the definition: 
    \begin{align}\label{equ:max-codeg-B122}
        %\Delta_2(B_{122}^{s}) \le \delta_{4} n 
        d_{B_{122}^{s}}(uv)\le \delta_{4} n,\quad\mbox{for every $uv\in V_1\times V_2$.} 
    \end{align}

    The two conclusions in the following claim follow directly from the $C_{4}^{3}$-freeness and $C_{5}^{3}$-freeness of $\mathcal{H}$, respectively. 
    \begin{claim}\label{CLAIM:bad-contribute-missing}
        The following statements hold. 
        \begin{enumerate}[label=(\roman*)]
        \item\label{it:bad-triple-B111} For every triple $abc \in B_{111}$ and every vertex $d \in V_2$, we have $\{abd, acd, bcd\} \cap M \neq \emptyset$.
        \item\label{it:bad-triple-B122} 
        Let $x_1x_2x_3 \in B_{122}$ be a bad triple with $x_2\in V_1$. Suppose that $\{x_4,x_5\} \subseteq V_1$ are two vertices satisfying $\{x_4x_5x_1, x_3x_4x_5\} \subseteq \mathcal{H}$. Then $\{x_2x_3x_4, x_5x_1x_2\} \cap M \neq\emptyset$.
    \end{enumerate}
    \end{claim}

    Next, we will compare the size of each of $B_{111}$, $B_{122}^{s}$, $B_{122}^{\heavy}$ with the size of~$M$. First, we establish the following upper bound for the maximum degree of $B$ and~$M$. 
    
    \begin{claim}\label{CLAIM:max-deg-bad-triples}
        We have $\max\left\{ \Delta(B),\,\Delta(M) \right\} \le \delta_{3} n^2$.  
    \end{claim}
    \begin{proof}[Proof of Claim~\ref{CLAIM:max-deg-bad-triples}]
        By Claim~\ref{CLAIM:B-M-max-deg}, it suffices to show that $\Delta(B) \le \delta_{3} n^2/2$. 
         Fix a vertex $v \in V$ of maximum degree in $B$ and let $G \coloneqq L(v)$. 
         Note from~\eqref{equ:exact-min-max-deg} that 
         \begin{align}\label{equ:link-v-upper-lower-bound}
            \left(\frac{\alpha}{2} - \delta \right) n^2
             \le |G|
             \le \left(\frac{\alpha}{2} + \delta \right) n^2
         \end{align}
         First, we show that the number of both type-1 and type-2 four-vertex paths in $G$ (with respect to $(V_1, V_2)$) is at most $2\delta_1 n^4$.

         Suppose that $(a,b,c,d)$ is a $4$-vertex path of type-1 in $G$. 
         Then it follows from the $C_{5}^{3}$-freeness of $\mathcal{H}$ that $\{abd, acd\} \cap M \neq \emptyset$. 
         In other words, the set $\{a,b,c,d\}$ contains at least one missing triple from $M$. 
         If there are at least $2\delta_1 n^4$ copies of $4$-vertex path of type-1 in $G$, then the number of missing triples is at least $2\delta_1 n^4/n = 2\delta_1 n^3$, contradicting~\eqref{equ:vtx-stab-b}. 
         Therefore, the number of type-1 paths in $G$ is at most $2\delta_1 n^4$.

         Now suppose that $(a,b,c,d)$ is a $4$-vertex path of type-2 in $G$. Then it follows from the $C_{5}^{3}$-freeness of $\mathcal{H}$ that $\{adb, adc\} \cap M \neq \emptyset$. 
         In other words, the set $\{a,b,c,d\}$ contains at least one missing triple from $M$. 
         If there are at least $2\delta_1 n^4$ copies of $4$-vertex path of type-2 in $G$, then the number of missing triples is at least $2\delta_1 n^4/n = 2\delta_1 n^3$, contradicting~\eqref{equ:vtx-stab-b}.  
         Therefore, the number of type-2 paths in $G$ is at most $2\delta_1 n^4$. 

         Combining the conclusion above with~\eqref{equ:vtx-stab-a} and~\eqref{equ:vtx-stab-d}, we see that $G$ satisfies all three assumptions in Proposition~\ref{PROP:P4-free} with $\delta$ there corresponding to $2\delta_1$ here. Moreover, by~\eqref{equ:link-v-upper-lower-bound}, we have $|G| \ge \left(\frac{\alpha}{2} - \delta\right)n^2$. 
         Therefore, by Proposition~\ref{PROP:P4-free}, either $|G[V_1]| + |G[V_2]| \le \delta_2 n^2$ or $|G[V_1, V_2]| \le \delta_2 n^2$. 
        
         Suppose that $|G[V_1]| + |G[V_2]| \le \delta_2 n^2$. Then we have 
         \begin{align*}
             |G[V_1,V_2]|
             = |G| - |G[V_1]| - |G[V_2]| 
             \ge \left(\frac{\alpha}{2} - \delta - \delta_{2}\right)n^2
             > \delta_2 n^2. 
         \end{align*}
         This implies  that $v \in V_1$, since otherwise, it would contradict the  local maximality constraint~\eqref{equ:local-max-link-V2}. Therefore, $d_{B}(v) = |G[V_1]| + |G[V_2]| \le \delta_2 n^2 \le \delta_{3} n^2/2$, as desired.  

         Suppose that $|G[V_1, V_2]| \le \delta_2 n^2$. Then combining it with~\eqref{equ:vtx-stab-a}, we have 
         \begin{align*}
             |G[V_1]| 
             = |G| - |G[V_1, V_2]| - |G[V_2]| 
             \ge \left(\frac{\alpha}{2} - \delta\right)n^2 - \delta_{2} n^2 - \frac{(1-\gamma + \delta)^2n^2}{2}
             > \delta_2 n^2. 
         \end{align*}
         This implies that $v \in V_2$, since otherwise, it would contradict the  local maximality constraint~\eqref{equ:local-max-link-V1}. Therefore, $d_{B}(v) = |G[V_1, V_2]| \le \delta_2 n^2 \le \delta_{3} n^2/2$, also as desired.
         This completes the proof of Claim~\ref{CLAIM:B-M-max-deg}. 
    \end{proof}%CLAIM

    \begin{claim}\label{CLAIM:B122-vs-M}
        We have $|B_{122}^{s}| \le \delta_{5} |M|$. 
    \end{claim}
    \begin{proof}[Proof of Claim~\ref{CLAIM:B122-vs-M}]
        Let $abc \in B_{122}^{s}$ be a bad triple with $a\in V_1$ (and thus, $\{b,c\} \subseteq V_2$). Let $G$ denote the induced subgraph of the common link $L(b) \cap L(c)$ on $V_1$. Since, by Claim~\ref{CLAIM:max-deg-bad-triples}, $\max\{d_{M}(b), d_{M}(c) \} \le \delta_{3} n^2$, we have 
        \begin{align*}
            |G|
            \ge \binom{|V_1|}{2} - \left(d_{M}(b) + d_{M}(c)\right)
            \ge \binom{|V_1|}{2} - 2\delta_{3} n^2. 
        \end{align*}
        Let $\{d_1e_1, \ldots, d_te_t\}$ be a matching of maximum size in $G$. 
        It follows from Fact~\ref{FACT:matching-number-lower-bound} that 
        \begin{align*}
            t 
            \ge \frac{|G|}{2|V_1|} 
            \ge \frac{\binom{|V_1|}{2} - 2\delta_{3} n^2}{2|V_1|}
            \ge \frac{|V_1|}{5}. 
        \end{align*}
        Fix an arbitrary $i \in [t]$. Since $d_i e_i \in G \subseteq L(b) \cap L(c)$, we have $\{bd_ie_i, cd_ie_i\} \subseteq \mathcal{H}$. 
        It follows from Claim~\ref{CLAIM:bad-contribute-missing}~\ref{it:bad-triple-B122} that $\{abd_i, ace_i\} \cap M \neq \emptyset$. 
        This implies that $d_{M}(ab) + d_{M}(ac) \ge t \ge |V_1|/5$. 
        By the Pigeonhole Principle, we have $\max\left\{d_{M}(ab), d_{M}(ac) \right\} \ge |V_1|/10$. 

        For every bad triple $E \in B_{122}^{s}$, let $\partial E$ denote the pair inside $E$ with the larger missing degree among the two crossing pairs in $E$ (if there is a tie, break it arbitrarily).
        Let 
        \begin{align*}
            \partial B_{122}^{s} \coloneqq \left\{\partial E \colon E\in B_{122}^{s} \right\}\subseteq V_1\times V_2. 
        \end{align*}
        Note from the argument above that $d_{M}(\partial E) \ge |V_1|/10$ for every $E \in B_{122}^{s}$. 
        Combining it with~\eqref{equ:max-codeg-B122}, we obtain 
        \begin{align*}
            |B_{122}^{s}|
            \le \sum_{uv\in \partial B_{122}^{s}} d_{B_{122}^{s}}(uv) 
            \le \sum_{uv\in \partial B_{122}^{s}} \delta_4 n
            \le \frac{\delta_{4} n}{|V_1|/10} \sum_{uv\in \partial B_{122}^{s}} d_{M}(uv)
            \le \delta_{5} |M|. 
        \end{align*}
        This completes the proof of  Claim~\ref{CLAIM:B122-vs-M}. 
    \end{proof}%CLAIM
    
    \begin{claim}\label{CLAIM:max-codeg-B111}
        We have $\Delta_2(B_{111}) \le \delta_{4} n$ and $|B_{111}| \le \delta_{5} |M|$. 
    \end{claim}
    \begin{proof}[Proof of Claim~\ref{CLAIM:max-codeg-B111}]
        First, we prove that $\Delta_2(B_{111}) \le \delta_{4} n$. 
        Recall that $B_{111}$ is a $3$-graph on $V_1$. 
        Suppose to the contrary that there exists a pair $\{a,b\} \subseteq V_1$ such that $d_{B_{111}}(ab) \ge \delta_{4} n$. 
        Then fix a subset $N \subseteq N_{B_{111}}(ab) \subseteq N_{\mathcal{H}}(ab) \cap V_1$ of size $\delta_{4} n$. 
        
        Fix an arbitrary vertex $c \in V_2$, let $G$ denote the induced subgraph of $L(c)$ on $N$. Let $\{d_1e_1, \ldots, d_te_t\}$ be a matching of maximum size in $G$.  
        Note from Claim~\ref{CLAIM:max-deg-bad-triples} that 
        \begin{align*}
            |G| 
            \ge \binom{|N|}{2} - d_{M}(c)
            \ge \binom{|N|}{2} - \delta_{3} n^2. 
        \end{align*}
        Hence, it follows from Fact~\ref{FACT:matching-number-lower-bound} that 
        \begin{align*}
            t 
            \ge \frac{|G|}{2|N|}
            \ge \frac{\binom{|N|}{2} - \delta_{3} n^2}{2|N|}
            \ge \frac{|N|}{5}. 
        \end{align*}
        Fix an arbitrary $i \in [t]$, noting from the definition of $N$ and $G$ that $\{abd_i, abe_i, cd_ie_i\} \subseteq \mathcal{H}$. 
        So it follows from the $C_{5}^{3}$-freeness of $\mathcal{H}$ that $\{acd_i, bce_i\} \cap M \neq \emptyset$. Thus, we have 
        \begin{align*}
            d_{M}(ac) + d_{M}(bc) 
            \ge t 
            \ge \frac{|N|}{5}. 
        \end{align*}
        Summing over all $c \in V_2$, we obtain 
        \begin{align*}
            d_{M}(a) + d_{M}(b)
            \ge \sum_{c\in V_2} \left(d_{M}(ac) + d_{M}(bc)\right)
            \ge \frac{|V_2||N|}{5}. 
        \end{align*}
        By symmetry, we may assume that 
        \begin{align*}
            d_{M}(a)
            \ge \frac{d_{M}(a) + d_{M}(b)}{2}
            \ge \frac{1}{2} \cdot \frac{|V_2||N|}{5}
            \ge \frac{(1-\gamma-\delta_{1})n \cdot \delta_{4} n}{10}
            > \delta_{3} n^2.
        \end{align*}
        However, this is a contradiction to Claim~\ref{CLAIM:max-deg-bad-triples}. This proves that $\Delta_{2}(B_{111}) \le \delta_{4} n$. 

        Now we prove that $|B_{111}| \le \delta_{5} |M|$. 
        For every edge $E = abc \in B_{111}$,  denote by $\partial E$ the pair in $E$ with the largest missing degree among the three pairs in $E$ (if there is a tie, break it arbitrarily). Let 
        \begin{align*}
            \partial B_{111} 
            \coloneqq \left\{\partial E \colon E \in B_{111}\right\}. 
        \end{align*}

        For every $abc\in B_{111}$, it holds by Claim~\ref{CLAIM:bad-contribute-missing}~\ref{it:bad-triple-B111} that   $\{abd, acd, bcd\} \cap M \neq \emptyset$ for every vertex $d \in V_2$; this implies that there are at least $|V_2|$ missing triples that share two vertices with~$abc$. Hence, it follows from the Pigeonhole Principle that $d_{M}(\partial E) \ge |V_2|/3$. 
        Combining it with the conclusion $\Delta_2(B_{111}) \le \delta_{4} n$ established above, we obtain 
        \begin{align*}
            |B_{111}|
            \le \sum_{uv \in \partial B_{111}} d_{B_{111}}(uv) 
            \le \frac{\Delta_{2}(B_{111})}{|V_2|/3} \sum_{uv \in \partial B_{111}} d_{M}(uv) 
            \le \frac{\delta_{4} n}{(1-\gamma - \delta_1)/3} |M|
            \le \delta_{5} |M|, 
        \end{align*}
        which completes the proof of Claim~\ref{CLAIM:max-codeg-B111}.
    \end{proof}%CLAIM

    Next, we consider $B_{122}^{\heavy}$. 
    Recall that a pair $(a,b) \in V_1 \times V_2$ was defined to be heavy if $d_{B}(ab) \ge \delta_{4} n$.
    
    \begin{claim}\label{CLAIM:heavy-pair-vs-M}
        Suppose that $(a,b) \in V_1 \times V_2$ is a heavy pair. Then $d_{M}(ab) \ge |V_1| - \delta_{4} n$.
    \end{claim}
    \begin{proof}[Proof of Claim~\ref{CLAIM:heavy-pair-vs-M}]
        Fix a heavy pair $(a,b) \in V_1 \times V_2$ and let $N_{i} \coloneqq N_{\mathcal{H}}(ab) \cap V_i$ for $i \in [2]$. 
        It follows from the definition of heavy that $|N_2| \ge \delta_{4} n$.
        Suppose to the contrary that $|N_1| \ge \delta_{4} n$. 
        
        Fix an arbitrary vertex $c \in N_2$. Let $G$ denote the induced subgraph of the common link $L(b) \cap L(c)$ on $N_1$. Let $\{d_1e_1, \ldots, d_te_t\}$ be a matching of maximum size in $G$.  
        Since, by Claim~\ref{CLAIM:max-deg-bad-triples}, $\max\{d_{M}(b), d_{M}(c) \} \le \delta_{3} n^2$, we have 
        \begin{align*}
            |G|
            \ge \binom{|N_1|}{2} - \left(d_{M}(b) + d_{M}(c)\right)
            \ge \binom{|N_1|}{2} - 2 \delta_{3} n^2. 
        \end{align*}
        It follows from Fact~\ref{FACT:matching-number-lower-bound} that 
        \begin{align*}
            t 
            \ge \frac{|G|}{2|N_1|}
            \ge \frac{\binom{|N_1|}{2} - 2 \delta_{3} n^2}{2|N_1|}
            \ge \frac{|N_1|}{5}. 
        \end{align*}
        Fix an arbitrary $i \in [t]$. 
        Since $\{abc, abd_i, bd_ie_i, cd_ie_i\} \subseteq \mathcal{H}$, it follows from the $C_{5}^{3}$-freeness of $\mathcal{H}$ that $ace_i \in M$. This implies that $d_{M}(ac) \ge t \ge |N_1|/5$.
        Summing over all $c \in N_2$, we obtain 
        \begin{align*}
            d_{M}(a) 
            \ge \sum_{c\in N_2} d_{M}(ac) 
            \ge \frac{|N_2||N_1|}{5}
            \ge \frac{\delta_{4}n \cdot \delta_{4}n}{5}
            > \delta_{3} n^2, 
        \end{align*}
        contradicting Claim~\ref{CLAIM:max-deg-bad-triples}. 
        This proves Claim~\ref{CLAIM:heavy-pair-vs-M}. 
    \end{proof}%CLAIM
    
    Define two subfamilies of $M$ as follows:
    \begin{align*}
        M_{1}
        & \coloneqq \left\{e\in M \colon \text{$e$ contains exactly one heavy pair}\right\} 
        \quad\text{and}\quad  \\
        M_{2}
        & \coloneqq \left\{e\in M \colon \text{$e$ contains exactly two heavy pairs}\right\}. 
    \end{align*}

    \begin{claim}\label{CLAIM:max-deg-heavy-pair}
        Every vertex is contained in at most $\delta_{4} n$ heavy pairs. Consequently, $d_{M_2}(ab) < \delta_{4}n$ for every heavy pair $(a,b) \in V_1 \times V_2$. 
    \end{claim}
    \begin{proof}[Proof of Claim~\ref{CLAIM:max-deg-heavy-pair}]
        Suppose  that we can pick $\delta_{4} n$ heavy pairs containing some vertex $a$, say pairs $(a,b_i)$ or $(b_i,a)$ for $ i \in [\delta_{4} n]$.
        Then it follows from Claim~\ref{CLAIM:heavy-pair-vs-M} that 
        \begin{align*}
            d_{M}(a)
            \ge \frac{1}{2} \sum_{i \in [\delta_{4} n]} d_{M}(ab_i)
            \ge \frac{1}{2}\cdot \delta_{4} n \cdot \left(|V_1| - \delta_{4} n \right)
            \ge \frac{1}{2}\cdot \delta_{4} n \cdot \left((\gamma - \delta_{1})n - \delta_{4} n \right)
            > \delta_{3} n^2. 
        \end{align*}
        This contradicts Claim~\ref{CLAIM:max-deg-bad-triples}, proving the first part of the lemma.

        Let $(a,b) \in V_1 \times V_2$ be a heavy pair. Observe that every missing triple in $M_{2}$ that contains the pair $\{a,b\}$ also contributes another heavy pair that contains $b$. 
        So it follows from the conclusion above that $d_{M_2}(ab)< \delta_{4} n$. 
        This completes the proof of Claim~\ref{CLAIM:max-deg-heavy-pair}. 
    \end{proof}%CLAIM 

    \begin{claim}\label{CLAIM:M1-vs-M2}
        We have $|M_2| \le \delta_{4} |M_1|$. 
    \end{claim}
    \begin{proof}[Proof of Claim~\ref{CLAIM:M1-vs-M2}]
        Define an auxiliary bipartite graph $Q[M_1, M_2]$ whose vertex set is $M_1 \cup M_2$ and a pair of triples $(e, e') \in M_1 \times M_2$ is adjacent in $Q$ iff $e$ and $e'$ share one heavy pair. 
        Let $x \coloneqq \max\left\{d_{Q}(E) \colon E\in M_1\right\}$ and $y \coloneqq \min\left\{d_{Q}(E) \colon E\in M_2\right\}$. 
        
        Let $abc \in M_1$ be a missing triple of degree exactly $x$ in $Q$. Suppose that $b\in V_2$ and $ab$ is a heavy pair. 
        By the definition of $Q$, every missing triple in $N_{Q}(abc) \subseteq M_{2}$ must contain the pair $\{a,b\}$, so by Claim~\ref{CLAIM:max-deg-heavy-pair},
        \begin{align}\label{equ:Q-M1-max-deg}
            x = |N_{Q}(abc)| = d_{M_{2}}(ab) \le \delta_{4} n. 
        \end{align}
        Now let $abc \in M_{2}$ be a missing triple of degree exactly $y$ in $Q$. Suppose that $b\in V_2$ (hence, $\{a,b\}$ and $\{c,b\}$ are heavy pairs). 
        Since, by Claim~\ref{CLAIM:heavy-pair-vs-M}, $\min\left\{d_{M}(ab),\,d_{M}(cb)\right\} \ge |V_1| - \delta_{4} n$, and by Claim~\ref{CLAIM:max-deg-heavy-pair}, $\max\left\{d_{M_2}(ab),\,d_{M_2}(cb)\right\} \le\delta_{4} n$, it follows that  
        \begin{align*}
            \min\left\{d_{M_{1}}(ab),\,d_{M_{1}}(cb)\right\}
            = \min\left\{d_{M}(ab) - d_{M_2}(ab),\,d_{M}(cb)-d_{M_2}(cb)\right\}
            \ge |V_1| - 2 \delta_{4} n.
        \end{align*}
        Therefore, 
        \begin{align*}
            y
            = d_{Q}(abc) 
            = d_{M_{1}}(ab) + d_{M_{1}}(cb)
            \ge 2|V_1| - 4 \delta_{4} n. 
        \end{align*}
        Combining it with~\eqref{equ:Q-M1-max-deg}, we obtain 
        \begin{align*}
            \frac{|M_2|}{|M_1|}
            \le \frac{x}{y}
            \le \frac{\delta_{4} n}{2|V_1| - 4 \delta_{4} n}
            \le \frac{\delta_{4} n}{2(\gamma - \delta_{1})n - 4 \delta_{4} n}
            \le \delta_{4}, 
        \end{align*}    
        which proves Claim~\ref{CLAIM:M1-vs-M2}. 
    \end{proof}%CLAIM

    Let $H\subseteq V_1\times V_2$ denote the collection of all heavy pairs. 
    It follows from Claims~\ref{CLAIM:heavy-pair-vs-M} and~\ref{CLAIM:M1-vs-M2} that 
    \begin{align*}%\label{equ:B122L-vs-M}
        |B_{122}^{\heavy}|
        \le \sum_{uv \in H} d_{B_{122}^{\heavy}}(uv)
        \le \sum_{uv \in H} |V_2|
        & \le \frac{|V_2|}{|V_1| - \delta_{4} n} \cdot \sum_{uv \in H} d_{M}(uv)  \\
        & = \frac{|V_2|}{|V_1| - \delta_{4} n} \cdot \left(|M_1| + 2|M_2|\right)  \\
        & \le \frac{(1-\gamma +\delta_{1})n}{(\gamma-\delta_{1})n - \delta_{4} n} \cdot (1+2\delta_{4})|M| 
        \le \frac{3}{5}|M|. 
    \end{align*}
    Combining it with Claims~\ref{CLAIM:B122-vs-M} and~\ref{CLAIM:max-codeg-B111}, we obtain  
    \begin{align*}
        |B|
        = |B_{111}| + |B_{122}^{s}| + |B_{122}^{\heavy}|
        \le \delta_{5} |M| + \delta_{5} |M| + \frac{3}{5}|M| 
        \le \frac{4}{5} |M|. 
    \end{align*}
    Observe that the $3$-graph $\mathcal{H}' \coloneqq (\mathcal{H} \cup M) \setminus B$ is also $\{C_{4}^{3}, C_{5}^{3}\}$-free, and has size 
    \begin{align*}
        |\mathcal{H}'|
        = |\mathcal{H}| - |B| + |M| 
        \ge |\mathcal{H}| + \frac{|M|}{5}
        = \mathrm{ex}(n, \{C_{4}^{3}, C_{5}^{3}\}) + \frac{|M|}{5}.  
    \end{align*}
    It follows from the maximality of $\mathcal{H}$ that $|M| =  0$ (and hence, $|B| = 0$ as well), which completes the proof of Theorem~\ref{THM:K4C5-first-level-exact}.  
\end{proof}

\section{Concluding remarks}\label{se:concluding}

We could not determine the Tur\'an density of the tight $5$-cycle alone. When one tries to prove a version of Proposition~\ref{pr:3}, not only is the SDP problem considerably larger (having $|\C F_6|=40,547$ constraints) but also the numerical results returned by computer suggest that the SDP is not strong enough to prove that $|B|\le |M|+o(n^3)$ (even when the coefficient of $|M|$ is exactly 1). So it seems that some new ideas are needed here.

While the stability property of Theorem~\ref{THM:C5K4-stability} gives a good starting point for proving a version of~\eqref{eq:O(1)} with error term $O(n^2)$ when forbidding $C_\ell^3$ for given $l\ge 7$ not divisible by 3, any proof would need to overcome some new technical difficulties (such that a vertex may have large bad degree). So we decided against pursuing this direction (also since the bounds in~\eqref{eq:O(1)} for $\{C_4^3, C_5^3\}$ are much more precise).

%%%%%%%%%%%%%%%%%%%%%%%%%%%%%%%%%%%%%%%%%%%%
% \section{Concluding remarks}\label{SEC:remark}
% %
%%%%%%%%%%%%%%%%%%%%%%%%%%%%%%%%%%%%%%%%%%%%
\section*{Acknowledgements}
Levente Bodn\'ar, Xizhi Liu and Oleg Pikhurko were supported by ERC Advanced Grant 101020255.
%%%%%%%%%%%%%%%%%%%%%%%%%%%%%%%%%%%%%%%%%%%%%%
%\bibliographystyle{alpha}%abbrv
\bibliography{C5K4}
%%%%%%%%%%%%%%%%%%%%%%%%%%%%%%%%%%%%%%%%%%%
%%%%%%%%%%%%%%%%%%%%%%%%%%%%%
\end{document}